\documentclass[12pt,a4paper]{amsart}

\title{On the Table of Marks of a Direct Product of Finite Groups.}

\author{Brendan Masterson}
\address{B.M.: Department of Design Engineering and Mathematics, Middlesex University, London, The Boroughs, London NW4 4BT, United Kingdom}
\email{b.masterson@mdx.ac.uk}

\author{G\"otz Pfeiffer}
\address{G.P.: School of Mathematics, Statistics and Applied Mathematics,
  National University of Ireland, Galway,  University Road,
  Galway, Ireland}
\email{goetz.pfeiffer@nuigalway.ie}

\date{\today}

\usepackage[margin=1in]{geometry}
\usepackage{mathrsfs}
\usepackage{tikz}
\usepackage{tikz-cd}
\usetikzlibrary{arrows, matrix}
\usepackage[euler-digits]{eulervm}
\usepackage{tabulary}
\usepackage{nicefrac}

\newtheorem{thm}{Theorem}[section]
\newtheorem{lem}[thm]{Lemma}
\newtheorem{prop}[thm]{Proposition}
\newtheorem{cor}[thm]{Corollary}

\theoremstyle{definition}
\newtheorem{defn}[thm]{Definition}
\newtheorem{rmk}[thm]{Remark}
\newtheorem{ex}[thm]{Example}

\newcommand{\Sub}[1]{\mathcal{S}_{#1}}
\newcommand{\Sec}[1]{\mathcal{Q}_{#1}}
\newcommand{\Mor}[1]{\mathcal{M}_{#1}}
\newcommand{\Atm}[1]{A_{#1}}

\newcommand{\Aut}{\mathrm{Aut}}
\newcommand{\Inn}{\mathrm{Inn}}
\newcommand{\Out}{\mathrm{Out}}

\newcommand{\id}{\mathrm{id}}
\newcommand{\op}{\mathrm{op}}
\newcommand{\im}{\mathop{\mathrm{im}}}
\newcommand{\diag}{\mathrm{diag}}

\newcommand{\Q}{\mathbb{Q}}
\newcommand{\Z}{\mathbb{Z}}

\newcommand{\CIM}{\mathbf{A}}
\newcolumntype{K}[1]{>{\centering\arraybackslash}p{#1}}

\begin{document}

\begin{abstract}
We present a method for computing the table of marks of a direct
product of finite groups.  In contrast to the character table of a direct product of two finite groups, its table of marks is not simply
the Kronecker product of the tables of marks of the two groups.
Based on a decomposition of the inclusion order on the subgroup lattice
of a direct product as a relation product of three smaller partial orders,
we describe the table of marks of the direct product essentially
as a matrix product of three class incidence matrices.  Each of these
matrices is in turn described as a sparse block diagonal matrix.
As an application, we use a variant of this matrix product to construct
a ghost ring and a mark homomorphism for the rational double Burnside
algebra of the symmetric group~$S_3$.
\end{abstract}



\maketitle

\section{Introduction}
\label{sec:introduction}
The table of marks of a finite group $G$ was first introduced by William Burnside in his book \emph{Theory of groups of finite order} \cite{Burnside1911}. This table characterizes the actions of $G$ on transitive $G$-sets, which are in bijection to the conjugacy classes of subgroups of $G$. Thus the table of marks provides a complete classification of the permutation representations of a finite group $G$ up to equivalence.

 The Burnside ring  $B(G)$ of $G$ is the Grothendieck ring of the category of finite $G$-sets. The table of marks of $G$ arises as the matrix of the mark homomorphism from $B(G)$ to the free $\Z$-module $\Z^r$, where $r$ is the number of conjugacy classes of subgroups of $G$. Like the character table, the table of marks is an important invariant of the group $G$. By a classical theorem of Dress \cite{Dress1969}, $G$ is solvable if
 and only if the prime ideal spectrum of $B(G)$ is connected, i.e., if $B(G)$ has no nontrivial idempotents, a property that can easily be derived from the table of marks of $G$.

 The table of marks of a finite group $G$ can be determined by
  counting inclusions between conjugacy classes of subgroups of $G$
  \cite{Pfeiffer1997}.  For this, the subgroup lattice of $G$ needs to be known.  As
  the cost of complete knowledge of the subgroups of $G$ increases
  drastically with the order of $G$ (or rather the number of prime
  factors of that order), this approach is limited to small
  groups.  Alternative methods for the computation of a table of marks
  have been developed which avoid excessive computations with the
  subgroup lattice of $G$.  This includes a method for computing the
  table of marks of $G$ from the tables of marks of its maximal
  subgroups \cite{Pfeiffer1997}, and a method for computing the table of
  marks of a cyclic extension of $G$ from the table of marks of $G$
  \cite{NaughtonPfeiffer2012}.

The purpose of this article is to develop tools for the computation of the table of marks of a direct product of finite groups $G_{1}$ and $G_{2}$. The obvious idea here is to relate the subgroup lattice of $G_{1} \times G_{2}$ to the subgroup lattice of $G_{1}$ and $G_{2}$, and to compute the table of marks of $G_{1} \times G_{2}$ using this relationship. Many properties of $G_{1} \times G_{2}$ can be derived from the properties of $G_{1}$ and $G_{2}$ with little or no effort at all. Conjugacy classes of elements of $G_{1} \times G_{2}$, for example, are simply pairs of conjugacy classes of $G_{1}$ and $G_{2}$. And the character table of $G_{1} \times G_{2}$ is simply the Kronecker product of the character tables of $G_{1}$ and $G_{2}$. However the relationship between the table of marks of $G_{1} \times G_{2}$ and the tables of marks of $G_{1}$ and $G_{2}$ is much more intricate.

A flavour of the complexity to be expected is already given by a classical result known as  Goursat's Lemma (Lemma~\ref{goursat}), according to which the subgroups of a direct product of finite groups $G_{1}$ and $G_{2}$ correspond to isomorphisms between sections of $G_{1}$ and $G_{2}$.
 This article presents the first general and systematic study of the subgroup lattice of a direct product of finite groups beyond Goursat's Lemma.  Only very special cases of such subgroup lattices have been considered so far, e.g., by Schmidt \cite{Schmidt1978} and Zacher \cite{Zacher1981}.

In view of Goursat's Lemma, it seems appropriate to first develop some theory for sections in finite groups.  Here, a section of a finite group $G$ is a pair $(P, K)$ of subgroups $P, K$ of $G$ such that $K$ is a normal subgroup of $P$.  We study sections by first defining a partial order $\leq$ on the set of sections of $G$ as componentwise inclusion of subgroups: $(P', K') \leq (P, K)$ if $P' \leq P$ and $K' \leq K$.  Now, if $(P', K') \leq (P, K)$, the canonical  homomorphism $P'/K' \to P/K$ decomposes as a product of three maps: an epimorhism, an isomorphism and a monomorphism.  We show that this induces a decomposition of the partial order $\leq$ as a product of three partial orders, which we denote by $\leq_K$, $\leq_{P/K}$, and $\leq_P$ for reasons that will become clear in Section \ref{sec:sections}.  Thus
  \[ {\leq}  =  {\leq_K} \circ {\leq_{P/K}} \circ {\leq_P}, \]
and this decomposition of the partial order is compatible with the conjugation action of $G$ on the set of its sections.

The description of subgroups of $G_1 \times G_2$ in terms of sections of $G_1$ and $G_2$ allows us to transfer the decomposition of the partial orders on the sections to the set of subgroups of $G_1 \times G_2$.  We will show in Section \ref{sec:subgroups} that, for subgroups $L \leq M$ of $G_1 \times G_2$, there exist unique intermediary subgroups $L'$ and $M'$ such that
  \[L \leq_P L' \leq_{P/K} M' \leq_K M,\]
where the partial orders $\leq_P$, $\leq_{P/K}$ and $\leq_K$ on the set of subgroups of $G_1 \times G_2$ are defined in terms of the corresponding relations on the sections of $G_1$ and $G_2$.  This gives a decomposition of the partial order $\leq$ on subgroups into three partial orders which is compatible with the conjugation action of $G_1 \times G_2$.  In Section \ref{sec:tableofmarks}, we will show as one of our main results that this yields a corresponding decomposition of the table of marks of $G$ as a matrix product of three class incidence matrices.  Individually, each of these class incidence matrices has a block diagonal structure which is significantly easier to compute than the subgroup lattice of $G_1 \times G_2$.

The rest of this paper is arranged as follows: In Section
\ref{sec:preliminaries} we collect some useful known results. In Section
\ref{sec:sections} we study the sections of a finite group $G$ and discuss
properties of the lattice of sections, partially ordered componentwise.  We
show how a decomposition of this partial order as a relation product of three
partial orders leads to a corresponding decomposition of the class incidence
matrix of the sections of $G$ as a matrix product. This section concludes
with a brief discussion of an interesting variant $\leq'$ of the partial
order on sections, and its class incidence matrix. Section
\ref{sec:morphisms} considers isomorphisms from sections of $G$ to a
particular group $U$ as subgroups of $G \times U$. We determine the structure
of the set of all such isomorphisms as a $(G, \Aut(U))$-biset.  In Section
\ref{sec:subgroups}, we study subgroups of $G_1 \times G_2$ as pairs of such
isomorphisms, one from a section of $G_{1}$ into $U$, and one from
$G_2$. This allows us to determine the structure of the set of all such
subgroups as a $(G_1 \times G_2, \Aut(U))$-biset. We also derive a
decomposition of the subgroup inclusion order of $G_{1} \times G_{2}$ as a
relation product of three partial orders from the corresponding decomposition
of the partial orders of sections from Section \ref{sec:sections}. In Section
\ref{sec:tableofmarks} we develop methods for computing the individual class
incidence matrices for each of the partial orders on subgroups and use these
matrices to compute the table of marks of $G_{1} \times G_{2}$, essentially
as their product. Finally, in Section \ref{sec:burnside} we present an
application of the theory. The double Burnside ring $B(G, G)$ of a finite
group $G$ is defined as the Grothendieck ring of transitive
$\left( G, G \right)$-bisets and, where addition is defined as disjoint union
and multiplication is tensor product. The double Burnside ring is currently
at the centre of much research and is an important invariant of the group
$G$, see
e.g. \cite{BoltjeDanz2012,BoltjeDanz2013,bouc13:_simpl_burns,ragnarsson13:_satur_burns}. Here
we study the particular case of $G = S_3$, and use our partial orders to
construct an explicit ghost ring and mark homomorphism for $\Q B(G, G)$, in
the sense of Boltje and Danz~\cite{BoltjeDanz2012}.\\


\noindent\textbf{Acknowledgement:} Much of the work in this article is based on the first author's PhD thesis (see \cite{Masterson2016}). This research was supported by the College of Science, National University of Ireland, Galway.


\section{Preliminaries}
\label{sec:preliminaries}

\subsection{Notation.} We denote the symmetric group of degree $n$ by $S_n$,
the alternating group of degree $n$ by $A_n$, and a cyclic group of order $n$
simply by $n$.

We use various forms of composition in this paper.
Group homomorphisms act from the right and are composed accordingly:
the product of $\phi \colon G_1 \to G_2$ and $\psi\colon G_2 \to G_3$
is $\phi \cdot \psi \colon G_1 \to G_3$, defined by $a^{\phi \cdot \psi}
= (a^{\phi})^{\psi}$, for $a \in G_1$, where $G_i$ is a group, $i = 1,2,3$.

The relation product of relations $R \subseteq X \times Y$ and
$S \subseteq Y \times Z$ is the relation
$S \circ R = \{(x, z) : (x, y) \in R \text{ and } (y, z) \in S \text{ for
  some } y \in Y\} \subseteq X \times Z$, where $X, Y, Z$ are sets.

In section~\ref{sec:sub-rel}, the product $L * M$ of subgroups $L \leq G_1 \times G_2$ and $M \leq G_2 \times G_3$
will be defined as $(M^{\op} \circ L^{\op})^{\op}$, where $R^{\op} = \{(y, x) : (x, y) \in R\}$ denotes the opposite of $R$.

\subsection{Subgroups as Relations.}\label{sec:sub-rel}
The following  classical result describes subgroups of a direct product as isomorphisms between section quotients.  Here, a section of a finite group $G$ is a pair $(P, K)$ of subgroups of $G$ so that $K \unlhd P$.

\begin{lem}[Goursat's Lemma, \cite{Goursat1889}] \label{goursat}
Let $G_1, G_2$ be groups. There is a bijective correspondence between the subgroups of the direct product $G_1 \times G_2$ and the isomorphisms of the form $\theta \colon P_1/K_1 \to P_2/K_2$,
where $(P_i, K_i)$ is a section of $G_i$, $i = 1,2$.
\end{lem}

\begin{proof}
Let $L \leq G_1 \times G_2$ and let $P_i \leq G_i$ be the projection
of $L$ onto $G_i$, $i=1,2$.  Then $L$ is a binary relation from $P_1$ to $P_2$.
Writing $a_1 L a_2$ for $(a_1, a_2) \in L$,
it is easy to see that
$\{\{ a_2 \in P_2 : a_1 L a_2 \} : a_1 \in P_1\}$ is a partition of
$P_2$ into cosets of the normal subgroup
$K_2 = \{ a_2 \in G_2 : 1 L a_2 \} $ of $P_2$.  Similarly, the sets
$\{ a_1 \in P_1 : a_1 L a_2 \}$, $a_2 \in P_2$, are cosets of a normal
subgroup $K_1$ of $P_1$.
The relation $L$ thus is difunctional, i.e., it establishes a
bijection $\theta$ between the section quotients $P_1/K_1$ and
$P_2/K_2$, which in fact is a group homomorphism.

Conversely, any isomorphism $\theta \colon P_1/K_1 \to P_2/K_2$
between sections $(P_i, K_i)$ of $G_i$, $i=1,2$, yields a relation
$\{(a_1, a_2) \in G_1 \times G_2 : (a_1 K_1)^{\theta} = a_2 K_2 \}$,
which in fact is a subgroup of $G_1 \times G_2$.
\end{proof}

If a subgroup $L$ corresponds to an isomorphism $\theta: P_1/ K_1 \rightarrow P_2/K_2$, then we write $p_i (L)$ for $P_i$ and $k_i (L)$ for $K_i$,  $i = 1,2$. We call the sections $(P_i, K_i)$ the \emph{Goursat sections} of $L$ and the isomorphism type of $P_i/K_i$ the \emph{Goursat type} of $L$. Finally, $L$ is called the \emph{graph} of $\theta$ and, conversely, $\theta$ is the \emph{Goursat isomorphism} of $L$.

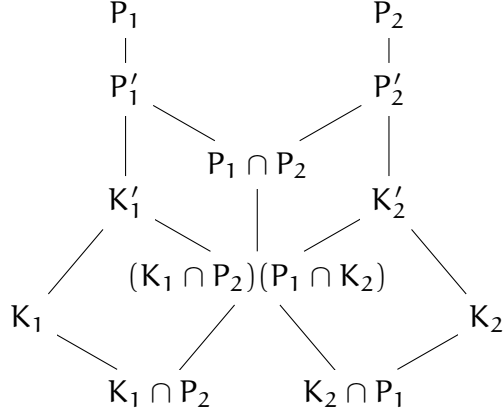
\begin{figure}[ht]
\begin{center}
\begin{tikzpicture}
\tikzstyle{every node} = [fill = white, minimum size = 4pt]

\draw (0, 0) node (1) {$(K_1 \cap P_2) (P_1 \cap K_2)$};
\draw (1) --++ (90:1.5cm) node (2) {$P_1 \cap P_2$};
\draw (1) --++ (230:2cm) node(3) {$K_1 \cap P_2$};
\draw (1) --++ (310:2cm) node (4) {$K_2 \cap P_1$};
\draw (3) --++ (150:2cm) node (5) {$K_1$};
\draw (4) --++ (30:2cm) node (6) {$K_2$};
\draw (1) --++ (30:2cm) node (7) {$K_2'$};
\draw (1) --++ (150:2cm) node (8) {$K_1'$};
\draw (7) --++ (90:1.5cm) node (9) {$P_2'$};
\draw (8) --++ (90:1.5cm) node (10) {$P_1'$};
\draw (10) --++ (90:1cm) node (11) {$P_1$};
\draw (9) --++ (90:1cm) node (12) {$P_2$};

\draw (9) --++ (2);
\draw (10) --++ (2);
\draw (5) --++ (8);
\draw (6) --++ (7);
\end{tikzpicture}
\end{center}
\caption{Butterfly Lemma}\label{fig:butterfly}
\end{figure}

The next lemma, illustrated in Fig.~\ref{fig:butterfly}, can be derived from Lemma~\ref{goursat}, see e.g. \cite{Lambek1958}.
\begin{lem}[Butterfly Lemma, \protect{\cite[11.3]{Huppert1967}}] \label{butters}
Let $(P_1, K_2)$ and $(P_2, K_2)$ be sections of $G$. Set $P'_{i}:= (P_1 \cap P_2) K_{i}$ for $i=1,2$, $K'_1 := (P_1 \cap K_2) K_1$, and $K_2' := (P_2 \cap K_1) K_2$.
Then $P_1 \cap P_2 = P_1' \cap P_2'$, $(K_1 \cap P_2)(P_1 \cap K_2) = K_1' \cap K_2'$ and the canonical map 
\begin{align*}
\phi_{i} : (P_1' \cap P_2') / (K_1' \cap K_2') \rightarrow P_{i}' / K_{i}'
\end{align*}
is an isomorphism, $i = 1,2$.
\end{lem}

We refer to the section $(P_1' \cap P_2', K_1' \cap K_2')$ as the \emph{Butterfly meet} of $(P_1, K_1)$ and $(P_2, K_2)$.

Let $G_1$, $G_2$, $G_3$ be
finite groups.  The product $L \ast M$ of subgroups
$L \leq G_1 \times G_2$ and $M \leq G_2 \times G_3$ is defined as
\begin{align*}
 L \ast M = \{(g_1, g_3) \in G_1 \times G_3 : (g_1, g_2) \in L \text{ and } (g_2,g_3) \in M \text{ for some } g_2 \in G_2 \}\text.
\end{align*}
Then  $L * M \subseteq G_1 \times G_3$ is in fact a subgroup.
Thanks to \cite{BoltjeDanz2013}, we obtain the Goursat isomorphism of $L \ast M$
by composing Goursat isomorphisms $\theta'$ and $\psi'$, as follows. Suppose
that $L$ is the graph of the isomorphism
$\theta \colon P_0/K_0 \to P_1/K_1$, and that $M$ is the graph of
$\psi \colon P_2/K_2 \to P_3/K_3$.  With both $(P_1,K_1)$ and $(P_2, K_2)$ being sections of $G_2$, let subgroups $P_i'$, $K_i'$, and isomorphisms
$\phi_i \colon (P_1' \cap P_2')/(K_1' \cap K_2') \to P_i'/K_i'$, $i=1,2$,
be as in the Butterfly Lemma~\ref{butters}.
Let
$\overline{\psi} \colon P_2'/K_2' \to P_3'/K_3'$ be the
isomorphism obtained by restricting
$\psi$ to $P_2'/K_2'$, defined by
$(p K_2')^{\overline{\psi}} = (p K_2')^{\psi} $ for $p \in P_2'$.
Moreover, let $\overline{\theta} \colon P_0'/K_0' \to P_1'/K_1'$ be the co-restriction of
$\theta$ to $P_1'/K_1'$, defined by
$(p K_0')^{\overline{\theta}} = (p K_0')^{\theta} $ for $p \in P_0'$.
Then the graph of
\begin{align*}
\theta' := \overline{\theta} \cdot \phi_1^{-1} \colon P_0'/K_0' \to (P_1' \cap P_2')/(K_1' \cap K_2')
\end{align*}
is a subgroup of $G_1 \times G_2$ (although not necessarily of $L$),
the graph of
\begin{align*}
\psi':= \phi_2 \cdot \overline{\psi} \colon (P_1' \cap P_2')/(K_1' \cap K_2') \to P_3'/K_3'
\end{align*}
is a subgroup of $G_2 \times G_3$.

\begin{lem}\label{lm:star_butter}
  With the above notation, $L * M$ is the graph of the composite
  isomorphism $\theta' \cdot \psi' \colon P_0'/K_0' \to P_3'/K_3'$.
\end{lem}
We use the subgroup product and its Goursat isomorphism in the proof of Theorem~\ref{sameu}.

\subsection{Bisets and Biset Products.}
The action of a direct product $G_1 \times G_2$ on a set $X$
is sometimes more conveniently described
as the two groups $G_i$ acting on the same set $X$, one from the left and one
from the right.
\begin{defn}[\protect{\cite[2.3.1]{Bouc2010}}]\label{def:biset}
Let $G_1$ and $G_2$ be groups. Then a $(G_1, G_2)$-\emph{biset} $X$ is a left $G_1$-set and a right $G_2$-set, such that the actions commute, i.e.,
\[
 (g_1 x) g_2 = g_1 (x  g_2), \quad g_i \in G_i,\, x \in X\text.
\]
\end{defn}

Under suitable conditions, bisets can be composed, as follows.
\begin{defn} \label{def:biset_comp}
Let $G_1$, $G_2$ and $G_3$ be groups. If $X$ is a $(G_1, G_2)$-biset and $Y$ a $(G_2, G_3)$-biset, the tensor product of $X$ and $Y$ is the $(G_1, G_3)$-biset
\[
X \times_{G_2} Y := (X\times Y)/G_2
\]
of $G_2$-orbits on the set $X \times Y$ under the action given by
$(x, y).g = (x.g, g^{-1}.y)$, $g \in G_2$.
\end{defn}
The tensor product of bisets will be used in Section~\ref{sec:subgroups} to
describe certain sets of subgroups of $G_1 \times G_2$.  It also provides
the multiplication in the double Burnside ring of a group $G$,
which is the subject of Section~\ref{sec:double-burnside}.

\subsection{Action on Pairs.}

We will also need to deal with one group acting on two sets.
The following parametrization of the orbits of a group acting on a set of
pairs is well-known.

\begin{lem}\label{la:orbits-of-pairs}
  Let $G$ be a finite group, acting on finite sets $X$ and $Y$, and suppose
  that $Z \subseteq X \times Y$ is a $G$-invariant set of pairs.  Then
  \begin{align*}
    Z/G = \coprod_{[y]_G \in Y/G} \{ [x, y]_G : [x]_{G_y} \in Zy/G_y \},
  \end{align*}
  where $Zy = \{x \in X : (x, y) \in Z\}$ for $y \in Y$.
\end{lem}

The $G$-orbits of pairs in $Z$ are thus represented by pairs $(x, y)$, where
the $y$ represent the orbits of $G$ on $Y$ and, for a fixed $y$, the $x$
represent the orbits of the stabilizer of $y$ on the set $Zy$ of all $x \in
X$ that are $Z$-related to $y$.

\begin{proof}
  Note that
  \begin{align*}
    Z = \coprod_{[y]_G \in Y/G} Z \cap (X \times [y]_G)
  \end{align*}
  is a disjoint union of $G$-invariant intersections $Z \cap (X \times
  [y]_G)$, whence $Z/G$ is the corresponding disjoint union of orbit spaces
  $(Z \cap (X \times [y]_G))/G$.  By~\cite[Lemma~2.1]{NaughtonPfeiffer2012} ,
  for each $y \in Y$, the map
  \begin{align*}
    [x]_{G_y} \mapsto [x, y]_G
  \end{align*}
  is a bijection between $X/G_y$ and $(X \times [y]_G)/G$.  Hence, for every
  $y \in Y$, there is a bijection between $Zy/G_y$ and
  $(Z \cap (X \times [y]_G))/G$,
\end{proof}

\subsection{Class Incidence Matrices.}
\label{sec:class-incidence-matrix}

Let $(X, \leq)$ be a finite
partially ordered  set (poset)
with incidence matrix
\begin{align*}
  A(\leq) = (a_{xy})_{x, y \in X}, \quad \text{where }
  a_{xy} =
  \begin{cases}
    1\text, & \text{if } y \leq x\text,\\
    0\text, & \text{else.}
  \end{cases}
\end{align*}
This incidence matrix $A(\leq)$ is lower triangular, if the order of rows and
columns of $A(\leq)$ extends the partial order $\leq$.

Suppose further that $\equiv$ is an equivalence relation on $X$.
Then $\equiv$ partitions $X$ into classes $X/{\equiv} = \{[x] : x \in T\}$, for a transversal $T \subseteq X$.
We say that the partial order $\leq$ is \emph{compatible} with the equivalence relation $\equiv$
if, for all classes $[x], [y]$,
the number
\begin{align*}
  \mathbf{a}_{xy} := \#\{x' \equiv x : y \leq x'\}
\end{align*}
does not depend on the choice of the representatives $x, y \in X$,
i.e., if $\mathbf{a}_{xy} = \mathbf{a}_{xy'}$ for $y' \equiv y$.
In that
case, we  define the \emph{class incidence matrix} of the partial
order $\leq$ to be the matrix
\begin{align*}
  \mathbf{A}(\leq) = (\mathbf{a}_{xy})_{x, y \in T}\text,
\end{align*}
whose rows and columns are labelled by the chosen transversal $T$.
Matrix multiplication relates the
matrices $A(\leq)$ and $\mathbf{A}(\leq)$  in the following way.

\begin{lem}\label{classin}
Define a row summing matrix $R(\equiv) = (r_{xy})_{x \in T,\, y \in X}$
and a column picking matrix
$C(\equiv) = (c_{xy})_{x \in X, y \in T}$ with entries
\begin{align*}
  r_{xy} &=
  \begin{cases}
    1\text, & \text{if } x \equiv y\text,\\
    0\text,& \text{else,}
  \end{cases}
&  c_{xy} &=
  \begin{cases}
    1\text, & \text{if } x = y\text,\\
    0\text,& \text{else.}
  \end{cases}
\end{align*}
Then
  \begin{enumerate}
  \item $R(\equiv) \cdot C(\equiv) = I$, the identity matrix on $T$.
  \item $R(\equiv) \cdot A(\leq) = \mathbf{A}(\leq) \cdot R(\equiv)$.
  \item $\mathbf{A}(\leq) = R(\equiv) \cdot A(\leq) \cdot C(\equiv)$,
  \end{enumerate}
\end{lem}
\begin{proof}
  (i) For each $x, z \in T$, $\sum_{y \in X} r_{xy} c_{yz} = r_{xz}$.
  (ii) For each $x \in T$, $z \in X$, the $x,z$-entry of both matrices
  is equal to $\mathbf{a}_{xy}$, where $y \in T$ represents the class $z \in X$.
  (iii) follows from (ii) and (i).
\end{proof}

\begin{rmk}
Examples of compatible posets are provided by group actions.
Suppose that a finite group $G$ acts on a poset
$(X, \leq)$ in such a way that
\begin{align*}
  x \leq y \implies x.a \leq y.a
\end{align*}
for all $x, y \in X$ and all $a \in G$.  Then $X$ is called a
\emph{$G$-poset}.  The partial order $\leq$ is compatible with the
partition of $X$ into $G$-orbits
since
\begin{align*}
  \{x' \equiv x : y \leq x'\}.a = \{x' \equiv x : y.a \leq x'\}\text,
\end{align*}
for all $x, y \in X$.  We write $R(G)$ and $C(G)$
for $R(\equiv)$ and $C(\equiv)$ if the equivalence $\equiv$ is
given by a $G$-action.
\end{rmk}

\begin{rmk}\label{rmk:comp}
More generally, any square matrix $A$ with rows and columns indexed by a
set $X$ with an equivalence relation $\equiv$, after choosing a transversal
of the equivalence classes, yields a product
$R(\equiv) \cdot A \cdot C(\equiv)$.  We say that the matrix $A$ is
\emph{compatible} with the equivalence if this product does not depend on
the choice of transversal.

If the equivalence on $X$ is induced by the action of a group $G$ then the
matrix $A = (a_{xy})_{x, y \in X}$ is compatible if $a_{x.g,y.g} = a_{xy}$
for all $g \in G$.  Such matrices are the subject of
Proposition~\ref{pro:cim-mor} and Theorem~\ref{sameu}.
\end{rmk}

\subsection{The Burnside Ring and the Table of Marks}

The \emph{Burnside ring} $B(G)$ of a finite group $G$
is the Grothendieck ring of the category of finite $G$-sets,
that is the free abelian group with basis consisting of the
isomorphism classes $[X]$ of transitive $G$-sets $X$,
with disjoint union as addition and the Cartesian product
as multiplication.
Multiplication of transitive $G$-sets is described by Mackey's formula
~\cite[Lemma~1.2.11]{LuxPahlings2010}
\begin{align*}
  [G/A] \cdot [G/B]
  = \sum_{\coprod_d AdB = G} [G/(A^d \cap B)]\text.
\end{align*}
The rational Burnside algebra
$\Q B(G) = \Q \otimes_{\Z} B(G)$ is isomorphic to a direct sum of $r$ copies of $\Q$, one for each conjugacy class of subgroups of $G$, with
products of basis elements determined by the above formula.

The \emph{mark} of a subgroup $H$ of $G$ on a $G$-set $X$ is its number of
fixed points, $|X^H| = \# \{ x \in X: x.h = x \text{ for all } h \in H\}$.
Obviously, $|X^{H_1}| =|X^{H_2}|$ whenever $H_1$ and $H_2$ are conjugate
subgroups of $G$.  The map $\beta_G \colon B(G) \to \Z^r$ assigns to
$[X] \in B(G)$ the vector $(|X^{H_1}|, \dots, |X^{H_r}|) \in \Z^r$, where
$(H_1, \dots, H_r)$ is a transversal of the conjugacy classes of subgroups of
$G$.  In this context, the ring $\Z^r$ with componentwise addition and
multiplication, is called the \emph{ghost ring} of~$G$.  We have
\begin{align*}
  \beta_G([X \amalg Y]) &= \beta_G([X]) + \beta_G([Y]),&
  \beta_G([X \times Y]) &= \beta_G([X]) \cdot \beta_G([Y]),
\end{align*}
where the latter product is componentwise multiplication in $\Z^r$.  Thus
$\beta_G$ is a homomorphism of rings, called the \emph{mark homomorphism}
of~$G$.

The \emph{table of marks} $M(G)$ of $G$ is the $r \times r$-matrix with rows
$\beta_G([G/H_i])$, $i = 1, \dots, r$, the mark vectors of all transitive
$G$-sets, up to isomorphism.  Regarding $\beta_G$ as a linear map from
$\Q B(G)$ to $\Q^r$, the table of marks is the matrix of $\beta_G$ relative
to the natural basis $([G/H_i])$ of $\Q B(G)$ and the standard basis of
$\Q^r$.

As $|(G/H)^K| = |N_G(H) : H| \, \#\{H^g \geq K : g \in G\}$ for subgroups
$H, K \leq G$, the table of marks provides a compact description of the
subgroup lattice of $G$.  In fact
\begin{align*}
  M(G) = D \cdot \CIM(\leq)\text,
\end{align*}
where $D$ is the diagonal matrix with entries $|N_G(H_i) : H_i|$
and $\CIM(\leq)$ is the class incidence matrix of the group $G$
acting on its lattice of subgroups by conjugation.

\begin{ex}
  Let $G = S_3$.  Then $G$ has $4$ conjugacy classes of subgroups and
  \begin{align*}
M(G) = \tiny \left(
    \begin{array}{r|llll}
      G/1 & 6 & . & . & . \\
      G/2 & 3 & 1 & . & . \\
      G/3 & 2 & . & 2 & . \\
      G/G & 1 & 1 & 1 & 1 \\
    \end{array}
\right)
=
\left(
    \begin{array}{llll}
      6 & . & . & . \\
      . & 1 & . & . \\
      . & . & 2 & . \\
      . & . & . & 1 \\
    \end{array}
\right)
 \cdot
\left(
    \begin{array}{llll}
      1 & . & . & . \\
      3 & 1 & . & . \\
      1 & . & 1 & . \\
      1 & 1 & 1 & 1 \\
    \end{array}
\right)\text.
  \end{align*}
\end{ex}


\section{Sections}
\label{sec:sections}

Let $G$ be a finite group.  We denote by $\Sub{G}$ the set of subgroups of
$G$, and by
\begin{align*}
  \Sub{G}/G := \{ [H]_G : H \leq G \}
\end{align*}
the set of conjugacy classes of subgroups of $G$.
A \emph{section} of $G$ is a pair $(P, K)$ of subgroups of $G$ where $K \trianglelefteq P$. We call $P$ the \emph{top group} and $K$ the \emph{bottom group} of the section $(P, K)$. We refer to the quotient group $P/K$ as the \emph{quotient} of the section $(P, K)$. The \emph{isomorphism type} of a section is the isomorphism type of its quotient and the \emph{size} of a section is the size its quotient. We denote the
set of sections of $G$ by
\begin{align*}
\Sec{G} := \left\{ (P, K) : K \trianglelefteq P \leq G \right\}\text.
\end{align*}
The group $G$ acts on the set of pairs $\Sec{G}$ by conjugation.  In
Sections~\ref{sec:conj-class-sect} and~\ref{sec:automizers}, we classify the
orbits of this action and describe the automorphisms induced by the
stabilizer of a section on its quotient.
The partial order on $\Sub{G}$ induces a partial order on
the pairs in $\Sec{G}$. In Section~\ref{sec:sections-lattice}, we show that
this partial order is in fact a lattice, and how it can be decomposed as a product of three smaller
partial order relations.
In Section~\ref{sec:class-incidences},we determine the class incidence matrix of $\Sec{G}$ and show that
the decomposition of the partial order on $\Sec{G}$ corresponds to
a decomposition of the class incidence matrix of $\Sec{G}$ as
a matrix product of three class incidence matrices.
In Section~\ref{sec:sect-latt-revis}, we use the smaller partial orders to define a new partial order on $\Sec{G}$ that is consistent with the notion of size of a section.

\subsection{Conjugacy Classes of Sections.}
\label{sec:conj-class-sect}
 A finite group $G$ naturally acts on its sections through componentwise conjugation via
\[
(P, K)^{g} := (P^{g}, K^{g})\text,
\]
 where $(P, K) \in \Sec{G}$ and $g \in G$.
We write $\left[ P, K \right]_{G}$ for the conjugacy class of a section $(P, K)$ in $G$, and denote the set of all conjugacy classes of sections of $G$ by
\[
\Sec{G}/G := \left\{ \left[ P, K \right]_{G} : (P , K) \in \Sec{G} \right\}\text.
\]
The conjugacy classes of sections can be parametrized in different ways in
terms of simpler actions, as follows.

\begin{prop} \label{la:sec_cong}
Let $G$ and $\Sub{G}$ be as above.
 \begin{enumerate}
 \item For $P \leq G$, let $\Sub{G}^{\unlhd P} = \{K \in \Sub{G} : K \unlhd P\}$.
Then $(\Sub{G}^{\unlhd P}, \leq)$ is an $N_G(P)$-poset and
\begin{align*}
 \Sec{G}/G = \coprod_{[P] \in \Sub{G}/G} \{ [P, K]_G : [K] \in \Sub{G}^{\unlhd P}/N_G(P)\}.
\end{align*}
 \item For $K \leq G$, let $\Sub{G}^{K \unlhd} = \{P \in \Sub{G} : K \unlhd P\}$.
Then $(\Sub{G}^{K \unlhd}, \leq)$ is an $N_G(K)$-poset and
\begin{align*}
 \Sec{G}/G = \coprod_{[K] \in \Sub{G}/G} \{ [P, K]_G : [K] \in \Sub{G}^{K \unlhd}/N_G(K)\}.
\end{align*}
 \end{enumerate}
\end{prop}

\begin{proof}
(i) Note that $\Sec{G} \subseteq \Sub{G} \times \Sub{G}$
is a $G$-invariant set of pairs.  As the stabilizer of
$K \in \Sub{G}$ is $N_G(K)$, the result follows with Lemma \ref{la:orbits-of-pairs}.
(ii) Follows in a similar way.
\end{proof}

We write $U \sqsubseteq G$ for a finite group $U$ which is isomorphic to a subquotient of $G$.  We denote by $\Sec{G}(U)$ the set of sections of $G$ with isomorphism type $U$, and by
\[
  \Sec{G}(U)/G := \{ \left[P, K\right]_{G} \in \Sec{G}/G : P/K \cong U \}.
\]
its $G$-conjugacy classes. Naturally,
\begin{align*}
\Sec{G}/G = \coprod_{U \sqsubseteq G} \Sec{G}(U)/G\text.
\end{align*}
Each of the above three partitions of $\Sec{G}/G$ will be used in the sequel.

\subsection{Section Automizers}
\label{sec:automizers}

The \emph{automizer} of a subgroup $H$ in $G$ is the quotient group of the section $(N_{G} (H), C_{G} (H))$. The automizer of $H$ is isomorphic to the subgroup of $\Aut (H)$ induced by the conjugation action of $G$. Analogously, we define the automizer of a section as a section whose quotient is isomorphic to the subgroup of automorphisms induced by conjugation by $G$.

\begin{defn}\label{def:sec-autom}
Let $(P, K) \in \Sec{G}$ and set $N = N_{G} (K)$. Using the natural homomorphism
\begin{align*}
\phi : N \rightarrow N/K,
\quad n \mapsto \overline{n} = nK\text,
\end{align*}
we let $\overline{P} := \phi (P) = P/K$ and $\overline{N} := \phi (N) = N/K$.
We define the \emph{section normalizer} of $(P, K)$ to be
the inverse image
$$N_{G} (P, K):= \phi^{-1} (N_{\overline{N}} (\overline{P})),$$
the \emph{section centralizer} to be $$C_{G} (P, K) := \phi^{-1} (C_{\overline{N}} (\overline{P})),$$
 and the \emph{section automizer} to be the section
$$\Atm{G} (P, K) := (N_{G} (P, K), C_{G} (P, K)).$$
Moreover, we denote by $\Aut_{G} (P, K)$ the subgroup of $\Aut(P/K)$ of automorphisms induced by conjugation by $G$, see Fig.~\ref{fig:section-auts}.
\end{defn}

\begin{figure}[ht]
  \begin{tikzpicture}
    \node (K) at (-6,-1) {$K$};
    \node (P) at (-6,2) {$P$};
    \node (meet) at (-6,0) {$P \cap C_G(P,K)$};
    \node (C) at (-3,1) {$C_G(P,K)$};
    \node (N) at (-3,4) {$N_G(P,K)$};
    \node (join) at (-3,3) {$P\,C_G(P,K)$};
    \draw (K) -- (meet) -- (P) -- (join);
    \draw (meet) -- (C) -- (join) --(N);

    \node (KK) at (0,-1) {$1$};
    \node (PK) at (0,2) {$\overline{P}$};
    \node (meetK) at (0,0) {$\overline{P} \cap C_{\overline{N}}(\overline{P})$};
    \node (CK) at (3,1) {$C_{\overline{N}}(\overline{P})$};
    \node (NK) at (3,4) {$N_{\overline{N}}(\overline{P})$};
    \node (joinK) at (3,3) {$\overline{P}\,C_{\overline{N}}(\overline{P})$};
    \draw (KK) -- (meetK) -- (PK) -- (joinK);
    \draw (meetK) -- (CK) -- (joinK) --(NK);

    \node (1) at (6,1) {$1$};
    \node (I) at (6,3) {$\Inn(P/K)$};
    \node (AG) at (6,4) {$\Aut_G(P,K)$};
    \node (A) at (6,5) {$\Aut(P/K)$};
    \draw (1) -- (I) -- (AG) -- (A);

  \end{tikzpicture}
\caption{The section $(P,K)$ and its automorphisms}\label{fig:section-auts}
\end{figure}
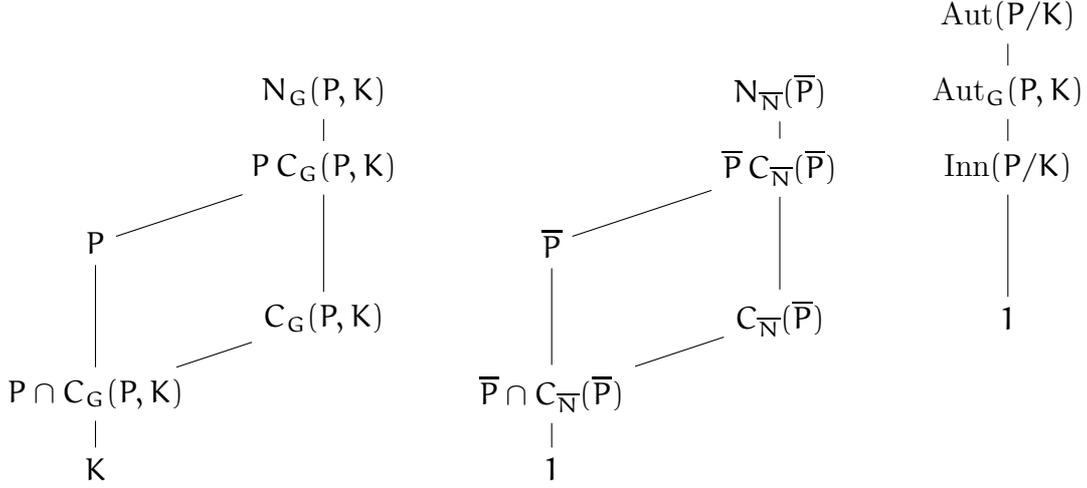

The following properties of these groups are obvious.

\begin{lem}\label{la:norm_cent} Let $(P, K)$ be a section in $\Sec{G}$. Then
\begin{enumerate}
\item $N_{G}(P, K) = N_{G}(P) \cap N_{G}(K)$.
\item $C_{G}(P, K)$ is the set of all $g \in N_{G}(P, K)$ which induce the identity automorphism on $P/K$.
\item $\Inn(G) \leq \Aut_{G}(P,K) \leq \Aut(P/K)$.
\end{enumerate}
\end{lem}

\subsection{The Sections Lattice.}
\label{sec:sections-lattice}

Subgroup inclusion induces a partial order on the set $\Sec{G}$ of
sections of $G$ which inherits the lattice property from the subgroup
lattice, as follows.

\begin{defn} \label{subsec}
$\Sec{G}$ is a poset, with partial order $\leq$ defined componentwise, i.e.,
  \begin{align*}
    (P', K') \leq (P, K) \text{ if } P' \leq P \text{ and } K' \leq K\text,
  \end{align*}
for sections $(P', K')$ and $(P, K)$  of $G$.
\end{defn}

For subgroups $A, B \leq G$, we write $A \vee B = \langle A, B \rangle$ for
the join of $A$ and $B$ in the subgroup lattice of $G$, and
$\langle\langle A \rangle\rangle_B$ for the normal closure of $A$ in $B$.

\begin{prop}\label{pro:sec-lattice}
  The poset $(\Sec{G}, \leq)$ is a lattice with componentwise meet, i.e.,
  \begin{align*}
    (P_1, K_1) \cap (P_2, K_2) = (P_1 \cap P_2, K_1 \cap K_2)\text,
  \end{align*}
  and join given by
  \begin{align*}
(P_1, K_1) \vee (P_2, K_2) = (P_1 \vee P_2, \left\langle \left\langle K_1 \vee K_2 \right\rangle \right\rangle_{P_1 \vee P_2})\text,
  \end{align*}
for sections $(P_1, K_1)$ and $(P_2, K_2)$ of $G$.
\end{prop}

\begin{proof}
Clearly, $K_1 \cap K_2$ is a normal subgroup of $P_1 \cap P_2$,
and the section $(P_1 \cap P_2, K_1 \cap K_2)$ is the unique greatest lower
bound of the sections $(P_1, K_1)$ and $(P_2, K_2)$ in $\Sec{G}$.

It is also easy to see that the least section $(P, K)$ of $G$ with
$P \geq P_1 \vee P_2$ and $K \geq K_1 \vee K_2$ has $P = P_1 \vee P_2$ and
$K = \langle\langle K_1 \vee K_2 \rangle\rangle_P$.
\end{proof}

\begin{thm} \label{secdecom}
Let $(P', K') \leq (P , K)$ be sections of a finite group $G$. Then
\begin{enumerate}
\item $(P', K\cap P')$  is the largest section between $(P', K')$ and $(P, K)$ with top group $P'$;
\item $(P'K, K)$ is the smallest section between $(P', K')$ and $(P, K)$ with bottom group $K$;
\item the map $p(K \cap P') \mapsto pK$, $p \in P'$ is an isomorphism between the section quotients of $(P', K \cap P')$ and $(P'K, K)$.
\end{enumerate}
\end{thm}

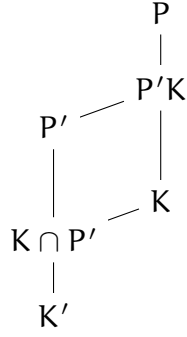
\begin{figure}[ht]
\begin{center}
\begin{tikzpicture}
\tikzstyle{every node} = [fill = white]
\draw (0,0) node (1) {$P$};
\draw (1) --++ (270:1cm) node (2) {$P'K$};
\draw (2) --++ (270:1.5cm) node (3) {$K$};
\draw (2) --++ (200:1.5cm) node (4) {$P'$};
\draw (3) --++ (200:1.5cm) node (5) {$K \cap P'$};
\draw (5) --++ (270:1cm) node (6) {$K'$};
\draw (4) --++ (5);
\end{tikzpicture}
\end{center}
\caption{$(P',K') \leq (P,K)$}\label{fig:sec-leq}
\end{figure}

\begin{proof}
If $(P', K') \leq (P, K)$ then there is a \emph{canonical homomorphism} from $P'/K'$ to $P/K$, given by $(K'p)^{\phi} =  Kp$ for $p \in P'$. According to the homomorphism theorem, $\phi$ can be decomposed into a surjective,  bijective and an injective part, that is $\phi = \phi_1 \phi_2 \phi_3$, where
\begin{align*}
\phi_1 & \colon P'/K' \rightarrow P'/K \cap P',&
\phi_2 & \colon P'/K\cap P' \rightarrow P'K / K,&
\phi_3 & \colon P'K / K \rightarrow P/K
\end{align*}
are uniquely determined, see Fig.~\ref{fig:sec-leq}
\end{proof}

Motivated by the above result we define the following three partial orders on $\Sec{G}$.

\begin{defn} \label{def:subpos}
Let $(P', K') \leq (P, K)$. Then we write
\begin{enumerate}
\item $(P', K') \leq_{P} (P, K)$ if $P' = P$, i.e., if the sections have the
  same top groups;
\item $(P', K') \leq_{K} (P, K)$ if $K' = K$, i.e., if the sections have the
  same bottom groups;
\item $(P', K') \leq_{P/K} (P, K)$ if the map $pK' \mapsto pK$, $p \in P$, is
  an isomorphism.
\end{enumerate}
\end{defn}

We can now reformulate Theorem \ref{secdecom} in terms of these three relations.
\begin{cor}\label{ince}
The partial order $\leq$ on $\Sec{G}$ is a product of three relations, i.e.,
$${\leq} = {\leq}_{K} \circ {\leq}_{P/K} \circ {\leq}_{P}\text.$$
Let $A(\leq)$ denote the incidence matrix for the partial order $\leq$. Then the stronger property
$$A(\leq) = A(\leq_{K}) \cdot A (\leq_{P/K}) \cdot A (\leq_{P})$$
also holds.
\end{cor}

\begin{proof}
By Theorem \ref{secdecom}, for $(P' , K') \leq (P, K)$ there exists unique intermediate sections $S, S' \in \Sec{G}$ such that $(P', K') \leq_{P} S' \leq_{P/K} S \leq_{K} (P, K)$.
\end{proof}

\begin{rmk}\label{rmk:corr-secs-subs}
  Note that, by the Correspondence Theorem, there is a bijective
  correspondence between the subgroups of $P/K$ and the sections $(P', K')$
  of $G$ with $(P', K') \leq_K (P, K)$.  Similarly, there is a bijective
  correspondence between the normal subgroups (and hence the factor groups)
  of $P/K$ and the sections $(P', K')$ of $G$ with $(P, K) \leq_P (P', K')$.
\end{rmk}

\subsection{Class Incidence Matrices.}
\label{sec:class-incidences}

We denote the class incidence matrix
of the $G$-poset $(\Sec{G}, \leq)$ by $\CIM(\leq)$.
Note that the set $\Sec{G}$ of sections of $G$ is
also a $G$-poset with respect to any of the
partial orders from Definition~\ref{def:subpos},
with respective class incidence matrices
$\CIM(\leq_P)$, $\CIM(\leq_K)$ and $\CIM(\leq_{P/K})$.

\begin{thm}\label{cimsecdecom}
With this notation,
\[
  \CIM(\leq) = \CIM(\leq_{K}) \cdot \CIM(\leq_{P/K}) \cdot \CIM(\leq_{P})\text.
\]
\end{thm}

\begin{proof}
Set $R = R(G)$ and $C = C(G)$.
From Lemma \ref{classin}(iii) we have that $\CIM(\leq) = R  \cdot A (\leq)  \cdot C$.  By Corollary \ref{ince} $A(\leq) = A (\leq_K)  \cdot A (\leq_{P/K})  \cdot A (\leq_P)$.  Lemma \ref{classin}(ii)  then gives
\begin{align*}
\CIM(\leq) &= R \cdot  A (\leq_K)  \cdot A (\leq_{P/K})  \cdot A (\leq_P)  \cdot C\\
&= \CIM(\leq_K)  \cdot R  \cdot A (\leq_{P/K})  \cdot A (\leq_P)  \cdot C\\
&= \CIM (\leq_K) \cdot \CIM (\leq_{P/K})  \cdot R  \cdot A (\leq_P) \cdot C\\
&= \CIM(\leq_K) \cdot \CIM(\leq_{P/K}) \cdot \CIM (\leq_P).
\end{align*}
\end{proof}

Each of the classes incidence matrices $\CIM(\leq_P)$, $\CIM(\leq_K)$ and
$\CIM(\leq_{P/K})$ is a direct sum of smaller class incidence matrices, as
the following results show.

\begin{thm}\label{secp}
For $P \leq G$, denote the class incidence matrix
of the $N_G(P)$-poset $\Sub{G}^{\unlhd P}$ by $\CIM_P(\leq)$.  Then
\[
\CIM (\leq_P) = \bigoplus_{[P] \in \Sub{G}/G} \CIM_P (\leq)\text.
\]
\end{thm}

\begin{proof}
  Let $(P, K) \in \Sec{G}$.  By Proposition~\ref{la:sec_cong}, the
  $G$-conjugacy classes containing a section with top group $P$ are
  represented by sections $(P, K')$, where $K'$ runs over a transversal of the
  $N_G(P)$-orbits of $\Sub{G}^{\unlhd P}$.  In order to count the
  $G$-conjugates of $(P, K')$ above $(P, K)$ in the $\leq_P$-order, it now
  suffices to note that $(P, K) \leq (P, K')^g$ for some $g \in G$ if and
  only if $K \leq (K')^g$ for some $g \in N_G(P)$.
\end{proof}

\begin{ex}\label{ex:s3-Psec}
Let $G = S_3$. Then
\[
\CIM (\leq_{P}) = \footnotesize \begin{array}{c|c|cc|cc|ccc}
  (1,1)& 1 &\cdot &\cdot &\cdot &\cdot &\cdot&\cdot &\cdot \\
  \hline
   (2, 1)&\cdot & 1 &\cdot &\cdot & \cdot &\cdot &\cdot &\cdot \\
    (2,2)&\cdot & 1 &1 & \cdot& \cdot &\cdot & \cdot&\cdot\\
    \hline
    (3,1)&\cdot & \cdot & \cdot & 1 & \cdot & \cdot& \cdot& \cdot \\
  (3,3)& \cdot & \cdot & \cdot & 1 & 1&\cdot&\cdot&\cdot \\
     \hline
  (G,1)& \cdot & \cdot & \cdot & \cdot & \cdot & 1&\cdot&\cdot  \\
  (G, 3)& \cdot &\cdot & \cdot & \cdot  & \cdot & 1& 1&\cdot \\
   (G,G)&\cdot&\cdot  &\cdot &\cdot & \cdot&1 &  1&1 \\
   \hline
   &(1, 1)&(2, 1)&(2, 2)&(3, 1)&(3, 3)&(G, 1)&(G, 3)&(G,G)
\end{array}
\]
\end{ex}

\begin{thm}\label{seck}
For $K \leq G$, denote the class incidence matrix of the $N_G(K)$-poset $\Sub{G}^{K \unlhd}$  by  $\CIM_K(\leq)$.  Then
\[
\CIM (\leq_{K}) = \bigoplus_{[K] \in \Sub{G}/G} \CIM_K (\leq)\text.
\]
\end{thm}

\begin{proof}
  Similar to the proof of Theorem~\ref{secp}.
\end{proof}

\begin{ex}\label{ex:s3-Ksec}
Let $G = S_3$.  Then
\[
\CIM (\leq_{K}) =  \footnotesize \begin{array}{c|cccc|c|cc|c}
  (1,1)& 1 &\cdot &\cdot &\cdot &\cdot &\cdot&\cdot &\cdot \\
   (2, 1)&3 & 1 &\cdot &\cdot & \cdot &\cdot &\cdot &\cdot \\
    (3,1)&1 & \cdot & 1 &\cdot &\cdot &\cdot & \cdot&\cdot\\
    (G,1)&1 & 1 & 1 & 1  & \cdot & \cdot& \cdot& \cdot \\
    \hline
  (2,2)& \cdot & \cdot & \cdot & \cdot & 1&\cdot&\cdot&\cdot \\
     \hline
  (3,3)& \cdot & \cdot & \cdot & \cdot & \cdot & 1&\cdot&\cdot  \\
  (G, 3)& \cdot &\cdot & \cdot & \cdot  & \cdot & 1& 1&\cdot \\
  \hline
   (G,G)&\cdot&\cdot  &\cdot &\cdot & \cdot&\cdot &  \cdot&1 \\
   \hline
   &(1, 1)&(2, 1)&(3, 1)&(G, 1)&(2, 2)&(3, 3)&(G, 3)&(G,G)
\end{array}
\]
\end{ex}

\begin{lem}
We have
\[
\CIM(\leq_{P/K}) = \bigoplus_{U \sqsubseteq G} \CIM_U(\leq_{P/K})
\]
where, for $U \sqsubseteq G$, $\CIM_U(\leq_{P/K})$ is the class incidence
matrix of the $G$-poset $(\Sec{G}(U), \leq_{P/K})$.
\end{lem}

\begin{proof}
  $(P',K') \leq_{P/K} (P, K)$ implies $P'/K' \cong P/K$.
\end{proof}

\begin{ex}\label{ex:s3-PKsec}
Let $G = S_3$.  Then
\[
\CIM(\leq_{P/K}) = \footnotesize\begin{array}{c|cccc|cc|c|c}
    (1,1)&1 &\cdot &\cdot &\cdot &\cdot &\cdot&\cdot &\cdot \\
   (2,2)& 3 & 1 &\cdot &\cdot &\cdot &\cdot &\cdot &\cdot \\
    (3,3)& 1 & \cdot & 1 &\cdot &\cdot &\cdot &\cdot&\cdot\\
    (G,G)& 1 & 1 & 1 & 1 &\cdot & \cdot& \cdot&\cdot \\
    \hline
   (2,1)& \cdot& \cdot & \cdot & \cdot & 1&\cdot&\cdot&\cdot \\
   (G,3)& \cdot & \cdot & \cdot & \cdot& 1 & 1&\cdot&\cdot  \\
   \hline
  (3,1)&  \cdot &\cdot & \cdot & \cdot & \cdot & \cdot &1&\cdot \\
  \hline
   (G,1)& \cdot &\cdot &\cdot &\cdot & \cdot&\cdot & \cdot& 1\\
   \hline
   &(1, 1)&(2, 2)&(3, 3)&(G,G)&(2, 1)&(G, 3)&(3, 1)&(G,1)
\end{array}
\]
The class incidence matrix $\CIM(\leq)$ of the $G$-poset
$(\Sec{G}, \leq)$ is the product of this matrix and the class incidence matrices
in Examples~\ref{ex:s3-Ksec} and \ref{ex:s3-Psec}, according to  Theorem~\ref{cimsecdecom}:
\[
\CIM(\leq) = \footnotesize\begin{array}{c|cccc|cc|c|c}
  (1,1)& 1 &\cdot &\cdot &\cdot &\cdot &\cdot&\cdot &\cdot \\
   (2,2)&3 & 1 &\cdot &\cdot & 1 &\cdot &\cdot &\cdot \\
    (3,3)&1 & \cdot & 1 &\cdot &\cdot &\cdot & 1&\cdot\\
    (G,G)&1 & 1 & 1 & 1  & 1 & 1& 1& 1 \\
    \hline
  (2,1)& 3& \cdot & \cdot & \cdot & 1&\cdot&\cdot&\cdot \\
  (G,3)& 1 & \cdot & 1 & \cdot & 1 & 1&1&1  \\
  \hline
  (3, 1)& 1 &\cdot & \cdot & \cdot  & \cdot & \cdot& 1&\cdot \\
  \hline
   (G,1)&1&\cdot  &\cdot &\cdot & 1&\cdot &  1&1 \\
   \hline
   &(1, 1)&(2, 2)&(3, 3)&(G,G)&(2, 1)&(G, 3)&(3, 1)&(G,1)
\end{array}
\]
\end{ex}

\subsection{The Sections Lattice Revisited.}
\label{sec:sect-latt-revis}

The partial order ${\leq} = {\leq_K} \circ {\leq_{P/K}} \circ {\leq_P}$ on
$\Sec{G}$ is not compatible with section size as
$(P',K') \leq_P (P, K)$ implies $|P'/K'| \geq |P/K|$.  It turns out
that, by effectively replacing the partial order $\leq_P$ by its opposite $\geq_P$, one
obtains from $\leq$ a new partial order $\leq'$, which is compatible with
section size.

\begin{prop} \label{altsubsec}
Define a relation $\leq'$ on $\Sec{G}$ by
\begin{align*}
(P',K') \leq' (P,K) \text{ if } P' \leq P \text{ and } K \cap P' \leq K'
\end{align*}
for sections $(P',K')$ and $(P,K)$ of $G$.
Then $(\Sec{G}, \leq')$ is a $G$-poset.
\end{prop}

\begin{proof}
  The relation $\leq'$ is clearly reflexive and antisymmetric on $\Sec{G}$,
  and compatible with the action of $G$.  Hence it only remains to be shown
  that this relation is transitive.

Let $(P'', K'')$, $(P', K')$ and $(P, K)$ be sections of $G$, such that $(P'', K'') \leq' (P', K')$ and $(P', K') \leq' (P, K)$. In order to show that $(P'', K'') \leq' (P,K)$, we need $P'' \leq P$ (which is clear), and $K \cap P'' \leq K''$. Intersecting both sides of $K \cap P' \leq K'$ with $P''$ gives  $K \cap P''  \leq  K' \cap P''  \leq  K''$, as desired.
\end{proof}

\begin{ex}\label{ex:v4}
  Let us denote the three subgroups of order $2$ of the Klein $4$-group
  $G = 2^2$ by $2_1$, $2_2$ and $2_3$.  Then
  $(2_1,1) \leq' (G, 2_2), (G, 2_3)$ and $(G, G) \leq' (G, 2_2), (G, 2_3)$.
  As the sections $(G, 2_2), (G, 2_3)$ have no unique infimum the poset
  $(\Sec{G}, \leq')$ is not a lattice.
\end{ex}

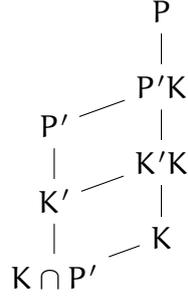
\begin{figure}[ht]
\begin{center}
\begin{tikzpicture}
\tikzstyle{every node} = [fill = white]
\draw (0,0) node (1) {$P$};
\draw (1) --++ (270:1cm) node (2) {$P'K$};
\draw (2) --++ (270:1cm) node (7) {$K'K$};
\draw (7) --++ (270:1cm) node (3) {$K$};
\draw (2) --++ (200:1.5cm) node (4) {$P'$};
\draw (4) --++ (270:1cm) node (5) {$K'$};
\draw (5) --++ (270:1cm) node (6) {$K \cap P'$};
\draw (3) --++ (6);
\draw (7) --++ (5);
\end{tikzpicture}
\end{center}
\caption{$(P',K') \leq' (P,K)$}\label{fig:sec-leq1}
\end{figure}

\begin{prop} \label{altsecdecom}
Let $(P', K')$ and $(P , K)$ be sections of a finite group $G$, such that $(P' , K'
) \leq' (P, K)$. Then, there are uniquely determined sections of $G$, $(P , K) \geq'
(P_1, K_1) \geq' (P_2, K_2) \geq' (P', K')$ such that
\begin{enumerate}
\item $(P_1, K_1) \leq_K (P, K)$,
\item $(P', K') \geq_P (P_2, K_2)$,
\item $(P_2,K_2) \leq_{P/K} (P_1,K_1)$.
\end{enumerate}
\end{prop}

\begin{proof}
By definition, $(P',K') \leq' (P,K)$ implies $P' \leq P$ and
$K \cap P' \leq K'$, where
$K' \unlhd P'$ and $K \unlhd P$. Then, by the second isomorphism theorem,
$K \cap P'$ is a normal subgroup of $P'$, $P' K$ is a subgroup of $P$ such that $K \trianglelefteq P' K$ and
$(P'K)/K$ is isomorphic to $P' / (K \cap P')$. Hence $(P_1,K_1) = (P'K,K)$ and $(P_2,K_2) = (P', K \cap P')$
have the desired properties, see Fig.~\ref{fig:sec-leq1}.
\end{proof}

\begin{cor}
The partial order $\leq'$ on $\Sec{G}$ is a product of three relations, i.e.,
\begin{align*}
{\leq'} = {\leq_{K}} \circ {\leq_{P/K}} \circ {\geq_{P}}\text.
\end{align*}
Moreover, $A(\leq') = A(\leq_K) \cdot A(\leq_{P/K}) \cdot A(\geq_P)$
and $\CIM(\leq') = \CIM(\leq_K) \cdot \CIM(\leq_{P/K}) \cdot \CIM(\geq_P)$.
\end{cor}

\begin{ex} Let $G = S_3$.  Then
\[
\CIM(\leq') =
 \footnotesize
 \begin{array}{c|cccc|cc|c|c}
  (1,1)& 1 &\cdot &\cdot &\cdot &\cdot &\cdot&\cdot &\cdot \\
   (2,2)&3 & 1 &\cdot &\cdot & \cdot &\cdot &\cdot &\cdot \\
    (3,3)&1 & \cdot & 1 &\cdot &\cdot &\cdot & \cdot&\cdot\\
    (G,G)&1 & 1 & 1 & 1  & \cdot & \cdot& \cdot & \cdot \\
    \hline
  (2,1)& 3& 1 & \cdot & \cdot & 1&\cdot&\cdot&\cdot \\
  (G,3)& 1 & 1 & 1 & 1 & 1 & 1&\cdot&\cdot  \\
  \hline
  (3, 1)& 1 &\cdot & 1 & \cdot  & \cdot & \cdot& 1&\cdot \\
  \hline
   (G,1)&1& 1 & 1 & 1 & 1& 1 &  1&1 \\
   \hline
   &(1, 1)&(2, 2)&(3, 3)&(G,G)&(2, 1)&(G, 3)&(3, 1)&(G,1)
\end{array}
\]
\end{ex}

In contrast to the class incidence matrix $\CIM(\leq)$ in
Example~\ref{ex:s3-PKsec}, the matrix $\CIM(\leq')$ is lower triangular
when rows and columns are sorted by section size.
Moreover, $(P, K) \leq' (G, 1)$, for all sections $(P, K)$ of $G$.

\begin{rmk}
  Whenever $(P', K') \leq' (P, K)$, there is a \emph{canonical isomorphism}
  $\psi \colon P'/K' \to P'K/K'K$.  Let $\theta: P_1/K_1 \to P_2/K_2$ be the
  Goursat isomorphism of a subgroup of $G_1 \times G_2$ and suppose that
  $(P', K') \leq' (P_1, K_1)$.  The canonical isomorphism determines a unique
  \emph{restriction} of $\theta$ to a Goursat isomorphism
  $\theta_1' \colon P'/K' \to P_2'/K_2'$.  Similarly, for each section
  $(P', K') \leq' (P_2, K_2)$, there is a unique \emph{co-restriction} of
  $\theta$ to a Goursat isomorphism $\theta_2' \colon P_1'/K_1' \to P'/K'$.
  As the Butterfly meet $(P', K')$ of sections $(P_1, K_1)$ and $(P_2, K_2)$
  of a group $G$ satisfies $(P', K') \leq' (P_1, K_i)$, $i = 1,2$, by
  Lemma~\ref{lm:star_butter}, the product of subgroups with Goursat
  isomorphisms $\theta \colon P_0/K_0 \to P_1/K_1$ and
  $\psi \colon P_2/K_2 \to P_3/K_3$ is the composition of the restriction of
  $\theta$ and the co-restriction of $\psi$ to the Butterfly meet of
  $(P_1, K_1)$ and $(P_2, K_2)$.
\end{rmk}


\section{Morphisms}
\label{sec:morphisms}

Let $U$ be a finite group.  A \emph{$U$-morphism}
of $G$ is an isomorphism $\theta \colon P/K \to U$ between a section $(P, K)$
of $G$ and the group $U$.  The set
\[
\Mor{G}(U):= \{ \theta \colon P/K \to U \mid (P, K) \in \Sec{G}(U) \}
\]
of all $U$-morphisms of $G$ forms a $(G,\Aut(U))$-biset.  In
Section~\ref{sec:biset-u-morphisms}, we describe the set $\Mor{G}(U)/G$ of
$G$-classes of $U$-morphisms as an $\Out(U)$-set.  The identification of
$\Mor{G}(U)$ with certain subgroups of $G \times U$
in Section~\ref{sec:comparing-morphisms}
induces a partial order
on $\Mor{G}(U)$.  In Section~\ref{sec:partial-order-morphism-classes}, we compute
the class incidence matrix of this partial order.

\subsection{Classes of $U$-Morphisms.}
\label{sec:biset-u-morphisms}

Each $U$-morphism  $\theta \colon P/K \to U$ of $G$ induces an
isomorphism between the automorphism groups $\Aut(P/K)$ and $\Aut(U)$.  We
define the automizer of $\theta$ as an isomorphism between the quotient of
the automizer $\Atm{G}(P, K)$ of the section $(P, K)$ and the corresponding
subgroup of $\Aut(U)$.

\begin{defn}
  Given a $U$-morphism $\theta\colon (P, K) \to U$, denote
  $\tilde{P} = N_G(P,K)$ and $\tilde{K} = C_G(P,K)$ and let
  \begin{align*}
    A_{\theta}\leq \Aut(U)
  \end{align*}
  be the image of $\Aut_G(P/K)$ in $\Aut(U)$.  The \emph{automizer} of the
  $U$-morphism $\theta$ is the $A_{\theta}$-morphism
  \[
  \Atm{G}(\theta) \colon \tilde{P}/\tilde{K} \to A_{\theta}\text,
  \]
  that, for $n \in \tilde{P}$, maps the coset $n\tilde{K}$ to the
  automorphism $\theta^{-1} \gamma_n \theta$ of $U$ corresponding to conjugation by $n$ on $P/K$.  Moreover, denote by
\begin{align*}
  O_{\theta} := A_{\theta}/\Inn(U) \leq \Out(U)\text,
\end{align*}
  the group of \emph{outer} automorphisms of $U$ induced via $\theta$,
noting that $\Inn(U) \leq A_{\theta}$.
\end{defn}

The group $G$ acts on $\Mor{G}(U)$ via $\theta^a = \gamma_a^{-1} \theta$,
where $\gamma_a \colon P/K \to P^a/K^a$ is the conjugation map induced by
$a \in G$. We denote by
\begin{align*}
  [\theta]_G := \{\theta^a : a \in G\}
\end{align*}
the $G$-orbit of the $U$-morphism $\theta$ and by
\begin{align*}
\Mor{G}(U)/G := \{[\theta]_G : \theta \in \Mor{G}(U)\}
\end{align*}
the set of
$G$-classes of $U$-morphisms.

For a section
$(P, K) \in \Sec{G}(U)$, denote by
$\Mor{G}^{P,K}(U)$ the set of $U$-morphisms with domain $P/K$.
Under the action
$(\theta, \alpha) \mapsto \theta \alpha$, for $\theta \in \Mor{G}(U)$
and $\alpha \in \Aut(U)$, the set $\Mor{G}(U)$ decomposes
into regular $\Aut(U)$-orbits $\Mor{G}^{P,K}(U)$,
one for each section $(P, K) \in \Sec{G}(U)$.
As the action of $\Aut(U)$
commutes with that of $G$, it induces an $\Aut(U)$-action
$([\theta]_G, \alpha) \mapsto [\theta \alpha]_G$ on the set $\Mor{G}(U)/G$ of
$G$-classes.  This action can be used to classify the
$G$-classes of $U$-morphisms as follows.

\begin{prop}\label{actonmorph}
Let $U \sqsubseteq G$.
\begin{enumerate}
\item[(i)]
As $\Aut(U)$-set,
$\Mor{G}(U)/G$ is the disjoint union of transitive $\Aut(U)$-sets
\begin{align*}
  \Mor{G}^{P,K}(U)/G := \{[\theta]_G : \theta \in \Mor{G}^{P,K}(U)\}\text,
\end{align*}
one for each $G$-class of sections $[P, K]_G \in \Sec{G}(U)/G$.
\item[(ii)] Let $\theta \colon P/K \to U$ be a $U$-morphism
of $G$.  Then
\begin{align*}
 \Mor{G}^{P,K}(U)/G = \{ [\theta \alpha]_G : \alpha \in D_{\theta} \}\text,
\end{align*}
where $D_{\theta}$ is a transversal of the right cosets
$A_{\theta} \alpha$ of $A_{\theta}$ in $\Aut(U)$.
\end{enumerate}
\end{prop}

Note that, by an abuse of notation, $\Mor{G}^{P,K}(U)/G$ is the set of full $G$-orbits of the $U$-morphisms in $\Mor{G}^{P,K}(U)$, although $\Mor{G}^{P,K}(U)$ is not a $G$-set in general.

\begin{proof}
  Let $X = \Mor{G}(U)$ and let $Y = \Sec{G}(U)$.  Then $X$ can be identified
  with the $G$-invariant subset $Z$ of $X \times Y$ consisting of those pairs
  $(\theta, (P, K))$ where $\theta$ has domain $P/K$. By
  Lemma~\ref{la:orbits-of-pairs}, $Z/G$ is the disjoint union of
  $\Aut(U)$-orbits $\Mor{G}^{P,K}(U)/G$, one for each $G$-class $[P,K]_G$ of
  sections of $G$.

  Now let $\theta \colon P/K \to U$ be a $U$-morphism.  The stabilizer of the
  section $(P,K)$ in $G$ is its normalizer $N_G(P, K)$.  The automizer
  $A_G(\theta)$ transforms this action into the subgroup $A_{\theta}$ of
  $\Aut(U)$.  As $\Aut(U)$ acts regularly on $\Mor{G}^{P,K}(U)$, the
  $A_{\theta}$-orbits on this set correspond to the cosets of $A_{\theta}$ in
  $\Aut(U)$ and
  $\{[\theta]_G : \theta \in \Mor{G}^{P,K}(U)\} = \{[\theta \alpha]_G :
  \alpha \in D_{\theta}\}$.
\end{proof}

As $\Inn(U) \leq A_{\theta}$ for all $\theta \in \Mor{G}(U)$, the
$\Aut(U)$-action on $\Mor{G}(U)/G$ can be regarded as an $\Out(U)$-action.
Thus, for each section $(P,K)$ of $G$, the set
$\Mor{G}^{P,K}(U)/G$ is isomorphic to $\Out(U)/O_{\theta}$ as $\Out(U)$-set.

\begin{ex}
  Let $G = A_4$ and $U = 2^2$. Then $\Sec{G}(U) = \{( 2^2, 1)\}$, and
  $\Aut(U) \cong S_3$ makes two orbits on the $U$-morphisms of the form
  $\theta: 2^2/ 1 \rightarrow U$, as $A_{\theta} \cong 3$.
\end{ex}

\subsection{Comparing Morphisms}
\label{sec:comparing-morphisms}

By Goursat's Lemma (Lemma~\ref{goursat}),
 a $U$-morphism $\theta \colon P/K \to U$ corresponds to
the subgroup
\begin{align*}
 L = \{(p, (pK)^\theta) : p \in P\} \leq G \times U\text.
\end{align*}
We call $L$ the \emph{graph} of $\theta$.
The partial order on the subgroups of $G \times U$ induces a natural partial
order on $\Mor{G}(U)$, as follows: if $\theta$ and
$\theta'$ are $U$-morphisms with graphs
$L$ and
$L'$ then we define
\begin{align*}
\theta' \leq \theta : \iff L' \leq L\text.
\end{align*}
This partial order on $\Mor{G}(U)$ is closely related to the order $\leq_{P/K}$
on $\Sec{G}(U)$.

\begin{prop}\label{prop:comparing-morphisms}
  Let $\theta \colon P/K \to U$ and $\theta' \colon P'/K' \to U$ be
  $U$-morphisms of $G$.  Then
  \begin{center}
    $\theta' \leq \theta \iff (P',K') \leq_{P/K} (P,K)$ and
    $\theta' = \phi \theta$,
  \end{center}
  where $\phi \colon P'/K' \to P/K$ is the homomorphism defined by
  $(pK')^{\phi} = pK$ for $p \in P'$.
\end{prop}

\begin{proof}
  Let $L = \{(p, (pK)^\theta) : p \in P\}$ be the graph of $\theta$ and let
  $L' = \{(p, (pK')^{\theta'}) : p \in P'\}$ be that of $\theta'$.

  Assume first that $L' \leq L$.  This clearly implies $P' \leq P$ and
  $K' \leq K$.  Moreover, for any $p \in P'$, if
  $(p, (pK')^{\theta'}) \in L' \leq L$ then
  $(pK')^{\theta'} = (pK)^{\theta} = (pK')^{\phi \theta}$ as
  $(p, (pK)^{\theta})$ is the unique element in $L$ with first component $p$.
  Hence $\theta' = \phi \theta$.  Now $\phi = \theta' \theta^{-1}$ is an
  isomorphism, whence $(P',K') \leq_{P/K} (P,K)$.

  Conversely, if $(P',K') \leq (P,K)$ and $\theta' = \phi \theta$ then
  clearly
  $(p, (pK')^{\theta'}) = (p, (pK')^{\phi \theta}) = (p, (pK)^{\theta}) \in
  L$ for all $p \in P'$, whence $L' \leq L$.
\end{proof}

More generally, for finite groups $U, U' \subseteq G$, suppose that sections
$(P, K) \in \Sec{G}(U)$ and $(P', K') \in \Sec{G}(U')$ are such that
$(P', K') \leq (P, K)$ with canonical homomorphism
$\phi \colon P'/K' \to P/K$.  If $\theta \colon P/K \to U$ and
$\theta' \colon P'\!/K' \to U'$ are isomorphisms then the composition
\begin{align*}
  \lambda:= (\theta')^{-1} \phi \theta
\end{align*}
obviously is a homomorphism from $U'$ to $U$, see Fig.~\ref{fig:lambda}.
\begin{figure}[ht]
\begin{center}
\begin{tikzpicture}[>=latex]
\node (P1) at (-2,5) {$P/K$};
\node (P2) at (1,5) {$U$};
\node (P1') at (-2,3) {$P'\!/K'$};
\node (P2') at (1,3) {$U'$};

\draw[->] (P1) -- node[above] {$_\theta$} (P2);
\draw[->] (P1') -- node[above] {$_{\theta'}$} (P2');
\draw[->] (P1') -- node[right] {$_{\phi}$} (P1);
\draw[->] (P2') -- node[right] {$_{\lambda}$} (P2);
\end{tikzpicture}
\end{center}
\caption{$\lambda \colon U' \to U$}\label{fig:lambda}
\end{figure}

In case $U = U'$, the previous lemma says that $\theta' \leq \theta$ if and
only if $\lambda = \id_U$.  If $U \neq U'$ then $\theta$ and $\theta'$ are
incomparable.  However, there are the following connections to the partial
orders on $\Sec{G}$.

\begin{lem}\label{la:mor-corr}
  Let  $\theta \colon P/K \to U$ be a $U$-morphism. Then $\theta$ induces
  \begin{enumerate}
  \item[(i)] an order preserving bijection between the sections $(P',K')$ of
    $G$ with $(P',K') \leq_K (P, K)$ and the subgroups of $U$;
  \item[(ii)] an order preserving bijection between the sections $(P',K')$ of
    $G$ with $(P',K') \geq_P (P, K)$ and the normal subgroups of $U$.
  \end{enumerate}
\end{lem}

\begin{proof}
  This is an immediate consequence of Remark~\ref{rmk:corr-secs-subs} on the
  Correspondence Theorem.
\end{proof}

\subsection{The Partial Order of Morphism Classes.}
\label{sec:partial-order-morphism-classes}

The partial order $\leq$ on $\Mor{G}(U)$ is compatible in the sense of
Section~\ref{sec:class-incidence-matrix} with the conjugation action of $G$,
and hence yields a class incidence matrix
\begin{align*}
  \CIM^{G}_U (\leq) = \bigl(\mathbf{a}(\theta,\theta')\bigr)_{[\theta], [\theta'] \in \Mor{G}(U)/G},
\end{align*}
where, for $\theta, \theta' \in \Mor{G}(U)$,
\begin{align*}
  \mathbf{a}({\theta, \theta'}) = \# \{\theta^a \geq \theta' : a \in G \}.
\end{align*}
This matrix is a submatrix of the class incidence matrix of the subgroup
lattice of $G \times U$, corresponding to the classes of subgroups which
occur as graphs of $U$-morphisms.

\begin{prop}
  Suppose that $\theta, \theta' \in \Mor{G}(U)$ have graphs
  $L, L' \leq G \times U$.  Then
\begin{align*}
  \mathbf{a}({\theta,\theta'}) = \# \{L^{(a, u)} \geq L' : (a, u) \in G \times U \}\text.
\end{align*}
\end{prop}

\begin{proof}
The result follows if we can show that the
  $(G \times U)$-orbit of $L$ is not larger than its $G$-orbit.
For this,
  let $u \in U$.  Then $(p, u) \in L$ for some $p \in P$
and hence $L^{(p,u)} = L$.  But then $L^{(1,u)} = L^{(p^{-1}, 1)}$.
\end{proof}

As, for $\theta, \theta' \in \Mor{G}(U)$ and $\alpha \in \Aut(U)$, we have
\begin{align*}
 \theta' \leq \theta \iff \theta' \alpha \leq \theta \alpha,
\end{align*}
the matrix $\CIM^{G}_U (\leq)$ is compatible (in the sense of
Section~\ref{sec:class-incidence-matrix}) with the action of $\Out(U)$ on
$\Mor{G}(U)/G$.  In fact, this relates it to the class incidence matrix
$\CIM_U(\leq_{P/K})$ of $\Sec{G}(U)$ as follows.

\begin{prop}\label{pro:cim-mor}
With the row summing and column picking matrices corresponding to the $\Out(U)$-orbits on $\Mor{G}(U)/G$, we have
  \begin{align*}
    \CIM_U(\leq_{P/K}) = R(\Out(U)) \cdot \CIM^{G}_U (\leq) \cdot C(\Out(U))\text.
  \end{align*}
\end{prop}

\begin{proof}
  By Proposition~\ref{actonmorph}(i), the union of the classes
  $[\theta \alpha]_G$, $\alpha \in \Aut(U)$, is the set of all
  $U$-morphisms of the form $(P/K)^a \to U$ for some $a \in G$.  This set
  contains, for each conjugate $(P/K)^a$ with $(P/K)^a \geq_{P/K} P'/K'$,
  exactly one $U$-morphism above~$\theta' \colon P'/K' \to U$, by
  Proposition~\ref{prop:comparing-morphisms}.
\end{proof}

\begin{ex}\label{ex:a5-1}
  Let $G = A_5$ and $U = 3$.  Then $\Mor{G}(U)/G$ consists of three classes,
  one with $(P,K) = (3,1)$ and two with $(P, K) = (A_4,2^2)$, permuted by
  $\Out(U)$.  We have
\begin{align*}
  \CIM^G_U(\leq) &= \tiny\left(\arraycolsep3pt
  \begin{array}{ccc}
    1&&\\
    1&1&\\
    1&\cdot&1\\
  \end{array}
\right)\text,&
  \CIM_U(\leq_{P/K}) &
= \tiny
\left(\arraycolsep3pt
  \begin{array}{ccc}
    1&\cdot&\cdot\\
    \cdot &1&1\\
  \end{array}
\right)
\cdot
\left(\arraycolsep3pt
  \begin{array}{ccc}
    1&&\\
    1&1&\\
    1&\cdot&1\\
  \end{array}
\right)
\cdot
\left(\arraycolsep3pt
  \begin{array}{cc}
    1&\cdot\\
    \cdot&1\\
    \cdot&\cdot\\
  \end{array}
\right)
= \small\left(
  \begin{array}{cc}
    1&\\
    2&1\\
  \end{array}
\right)\text.
\end{align*}
\end{ex}


\section{Subgroups of a Direct Product}
\label{sec:subgroups}

From now on, let $G_1$ and $G_2$ be finite groups.
In this section we describe the subgroups and the conjugacy classes
of subgroups of the direct product $G_1 \times G_2$ in terms of
properties of the groups $G_1$ and $G_2$.
By Goursat's
Lemma~\ref{goursat}, the subgroups of  $G_1 \times G_2$
correspond to isomorphisms between sections of $G_1$ and $G_2$.
Any such isomorphism arises as composition of two $U$-morphisms, for a
suitable finite group $U$.  This motivates the study of subgroups of
$G_1 \times G_2$ as pairs of $U$-morphisms.

\subsection{Pairs of Morphisms.}
\label{sec:pairs-morphisms}

Let $U$ be a finite group.
We call $L = (\theta \colon  P_1/K_1 \to P_2/K_2)$ a \emph{$U$-subgroup} of $G_1 \times G_2$ if $U$ is its Goursat type, i.e., if $P_i/K_i \cong U$, $i = 1,2$, and we denote by
$\Sub{G_1 \times G_2}(U)$
the set of all $U$-subgroups of $G_1 \times G_2$.
Given morphisms $\theta_i \colon P_i/K_i \to U$ in $\Mor{G_i}(U)$, $i = 1,2$,
composition yields an isomorphism
  $\theta = \theta_1 \theta_2^{-1} \colon P_1/K_1 \to P_2/K_2$
with whose graph is a $U$-subgroup
$L \leq G_1 \times G_2$.  Hence there
is a map
  $\Pi: \Mor{G_1}(U) \times \Mor{G_2}(U) \to \Sub{G_1 \times G_2}(U)$
defined by
\begin{align*}
 \Pi(\theta_1, \theta_2) = \theta_1 \theta_2^{-1}.
\end{align*}
In fact, the $(G_1,G_2)$-biset $\Sub{G_1 \times G_2}(U)$ is the
tensor product of the $(G_1, \Aut(U))$-biset $\Mor{G_1}(U)$
and the opposite of the $(G_2, \Aut(U))$-biset $\Mor{G_2}(U)$.

\begin{prop}
  $\Sub{G_1 \times G_2}(U) = \Mor{G_1}(U) \times_{\Aut(U)}
  \Mor{G_2}(U)^{\op}$.
\end{prop}
\begin{proof}
  For any $U$-subgroup $L = (\theta \colon P_1/K_1 \to P_2/K_2)$ there exist
  $\theta_i \in \Mor{G_i}(U)$, $i = 1,2$, such that
  $\theta = \Pi(\theta_1, \theta_2)$.  Moreover, for
  $\theta_i, \theta_i' \in \Mor{G_i}(U)$,  we have
  $\Pi(\theta_1, \theta_2) = \Pi(\theta_1', \theta_2')$ if and only if
  $\theta_1' \theta_1^{-1} = \theta_2' \theta_2^{-1}$ in $\Aut(U)$.
\end{proof}

It will be convenient to express the order of a $U$-subgroup in terms of $U$.
\begin{lem}\label{la:size-subgroup}
  Let $L = (\theta \colon P_1/K_1 \to P_2/K_2)$ be a $U$-subgroup
  of $G_1 \times G_2$.  Then $|L| 
  = |K_1||K_2| |U| = |P_1||P_2|/|U|$.
\end{lem}

\subsection{Comparing Subgroups.}
\label{sec:comparing-subgroups}
Let $U$, $U'$ be finite groups.
We now describe and analyze the partial order of subgroups of $G_1 \times G_2$
in terms of pairs of morphisms.

\begin{prop} \label{prop:L'<=L}
  Let
  \begin{align*}
(\theta_i \colon P_i/K_i \to U) \in \Mor{G_i}(U)
  \quad \text{and} \quad (\theta'_i \colon P'_i/K'_i \to U') \in \Mor{G_i}(U')\text,\quad i = 1,2\text,
  \end{align*}
  be morphisms, let $\theta = \Pi(\theta_1, \theta_2)$,
  $\theta' = \Pi(\theta'_1, \theta'_2)$ with corresponding subgroups $L$,
  $L'$ of $G_1 \times G_2$.  Then $L' \leq L$ if and only if
  \begin{enumerate}
  \item $(P'_i, K'_i) \leq (P_i, K_i)$ as sections of $G_i$, $i = 1,2$; and
  \item $\lambda_1 = \lambda_2$, where
    $\lambda_i = (\theta'_i)^{-1} \phi_i \theta_i$,
    and $\phi_i \colon P'_i/K'_i \to P_i/K_i$ is the homomorphism
    defined by $(K'_ip)^{\phi_i} = K_i p$, for $p \in P'_i$, $i = 1,2$.
  \end{enumerate}
\end{prop}

\begin{figure}[ht]
\begin{center}
\begin{tikzpicture}[>=latex]
\node (P1) at (-4,5) {$P_1/K_1$};
\node (U1) at (-1,5) {$U$};
\node (P1') at (-4,3) {$P_1'/K_1'$};
\node (U1') at (-1,3) {$U'$};

\node (P2) at (4,5) {$P_2/K_2$};
\node (U2) at (1,5) {$U$};
\node (P2') at (4,3) {$P_2'/K_2'$};
\node (U2') at (1,3) {$U'$};

\draw[->] (P1) -- node[above] {$_{\theta_1}$} (U1);
\draw[->] (P1') -- node[above] {$_{\theta_1'}$} (U1');
\draw[->] (P1') -- node[right] {$_{\phi_1}$} (P1);
\draw[->] (U1') -- node[right] {$_{\lambda_1}$} (U1);

\draw[->] (P2) -- node[above] {$_{\theta_2}$} (U2);
\draw[->] (P2') -- node[above] {$_{\theta_2'}$} (U2');
\draw[->] (P2') -- node[right] {$_{\phi_2}$} (P2);
\draw[->] (U2') -- node[right] {$_{\lambda_2}$} (U2);
\end{tikzpicture}
\end{center}
\caption{$L' \leq L$}\label{fig:sub-leq}
\end{figure}

\begin{proof}
  Write
  $L' = \{ (p_1', p_2') \in P_1' \times P_2' : (p_1'
  K_1')^{\theta_1'} = (p_2' K_2')^{\theta_2'} \}$
  and
  $L = \{ (p_1, p_2) \in P_1 \times P_2 : (p_1 K_1)^{\theta_1} =
  (p_2 K_2)^{\theta_2} \}$.

  Then $L' \leq L$ if and only if $(P'_i, K'_i) \leq (P_i, K_i)$, $i = 1, 2$,
  and, for $p_i \in P_i'$, we have $(p_1 K_1)^{\theta_1} = (p_2 K_2)^{\theta_2}$.
  But if $p_i \in P_i'$ then
  \begin{align*}
    (p_i K_i)^{\theta_i} = (p_i K_i')^{\phi_i \theta_i}
    = (p_i K_i')^{\theta_i' \lambda_i}\text.
  \end{align*}
  So $(p_1 K_1)^{\theta_1} = (p_2 K_2)^{\theta_2}$ if and only if
  $\lambda_1 = \lambda_2$, see Fig.~\ref{fig:sub-leq}.
\end{proof}

\begin{cor} \label{cor:L'<=L}
With the notation of Proposition~\ref{prop:L'<=L},
  $L' \leq L$ if and only if
\begin{enumerate}
\item $(P'_i, K'_i) \leq (P_i, K_i)$ as sections of $G_i$, $i = 1,2$;
\item $\phi_1 \theta = \theta' \phi_2$
\end{enumerate}
\end{cor}

The partial orders on sections introduced in
Definition~\ref{def:subpos}
give rise to relations on the subgroups of $G_1 \times G_2$, as follows.

\begin{defn}\label{def:subpos-grp}
  Let $L = (\theta \colon P_1/K_1 \to P_2/K_2)$ and
  $L' = (\theta' \colon P'_1/K'_1 \to P'_2/K'_2)$ be subgroups of
  $G_1 \times G_2$ and suppose that $L' \leq L$.  We write
  \begin{itemize}
  \item[(i)] $L' \leq_P L$, if
    $(P'_i, K'_i) \leq_P (P_i, K_i)$, $i = 1,2$,\\
    i.e., if both sections of $L'$ and $L$ have the same top groups;
  \item[(ii)] $L' \leq_K L$, if
    $(P'_i, K'_i) \leq_K (P_i, K_i)$, $i = 1,2$,\\
    i.e., if both sections of $L'$ and $L$ have the same bottom groups;
  \item[(iii)] $L' \leq_{P/K} L$, if
    $(P'_i, K'_i) \leq_{P/K} (P_i, K_i)$, $i = 1,2$,\\
    i.e., if the canonical homomorphisms $\phi_i \colon P'_i/K'_i \to P_i/K_i$
    are isomorphisms.
  \end{itemize}
\end{defn}

All three relations are obviously partial orders.  Moreover, they decompose
the partial order $\leq$ on the subgroups of $G_1 \times G_2$, in analogy to
Corollary~\ref{ince}.

\begin{thm}\label{thm:intermediate-subgroups}
Let $L = (\theta \colon P_1/K_1 \to P_2/K_2)$ and
$L' = (\theta' \colon P_1'/K_1' \to P_2'/K_2')$ be such that $L' \leq L$.
Define a map $\hat{\theta}' \colon P'_1/(P'_1 \cap K_1) \to P'_2/(P'_2 \cap K_2)$
by $(p_1(P'_1 \cap K_1))^{\hat{\theta}'} = p_2(P'_2 \cap K_2)$, whenever $p_i \in P_i'$
are such that $(p_1 P_1')^{\theta'} = p_2 P_2'$,
and a map
$\tilde{\theta} \colon P_1'K_1/K_1 \to P_2'K_2/K_2$
by $(p_1K_1)^{\tilde{\theta}} = p_2K_2$
whenever $p_i \in P_i'$ are such that $(p_1K_1)^{\theta} = p_2K_2$.  Then
\begin{itemize}
\item[(i)] $\hat{\theta}'$ and $\tilde{\theta}$ are isomorphisms with
  corresponding graphs $L_{\hat{\theta}'}$ and
  $L_{\tilde{\theta}} \leq G_1 \times G_2$.
\item[(ii)] $L_{\hat{\theta}'}$ and $L_{\tilde{\theta}}$
are the unique subgroups of $G_1 \times G_2$ with
$L' \leq_P L_{\hat{\theta}'} \leq_{P/K} L_{\tilde{\theta}} \leq_K L$.
\end{itemize}
\end{thm}

\begin{proof}
  Denote by $\phi_i \colon P_i'/K_i' \to P_i/K_i$ the canonical homomorphism,
  $i = 1,2$.  Then, as in the proof of Theorem~\ref{secdecom}, $\phi_i$ is
  the product of an epimorphism
  $\phi_{i1} \colon P_i'/K_i' \to (P_i'/K_i')/\ker \phi_i$, an isomorphism
  $\phi_{i2} \colon (P_i'/K_i')/\ker \phi_i \to \im \phi_i$, and a
  monomorphism $\phi_{i3} \colon \im \phi_i \to P_i/K_i$.  By
  Corollary~\ref{cor:L'<=L}, $\phi_1 \theta = \theta' \phi_2$.  It follows
  that $(\im \phi_1)^\theta = \im \phi_2$ and
  $(\ker \phi_1)^{\theta'} = \ker \phi_2$.  Thus $\theta$ restricts to an
  isomorphism $\tilde{\theta}$ from $\im \phi_1$ to $\im \phi_2$, and
  $\theta'$ induces an isomorphism $\hat{\theta}'$ from
  $(P_1'/K_1')/\ker \phi_1$ to $(P_2'/K_2')/\ker \phi_2$, and the following
  diagram commutes.

\begin{figure}[ht]
\begin{center}
\begin{tikzpicture}[>=latex]
\node (P1) at (-4,3) {$P_1/K_1$};
\node (P2) at (4,3) {$P_2/K_2$};
\node (K1) at (-2,-1) {$(P'_1/K'_1)/\ker \phi_1$};
\node (K2) at (2,-1) {$(P'_2/K'_2)/\ker \phi_2$};
\node (P1') at (-4,-3) {$P'_1/K'_1$};
\node (P2') at (4,-3) {$P'_2/K'_2$};
\node (PK1) at (-2,1) {$\im \phi_1$};
\node (PK2) at (2,1) {$\im \phi_2$};

\draw[->] (P1') -- node[left] {${\phi_1}$} (P1);
\draw[->] (P2') -- node[right] {${\phi_2}$} (P2);

\draw[->] (P1) -- node[above] {${\theta}$} (P2);
\draw[->] (P1') -- node[above] {${\theta'}$} (P2');
\draw[->] (K1) -- node[above] {${\hat{\theta}'}$} (K2);
\draw[->] (PK1) -- node[above] {${\tilde{\theta}}$} (PK2);

\draw[->] (P1') -- node[left] {${\phi_{11}}$} (K1);
\draw[->] (P2') -- node[right] {${\phi_{21}}$} (K2);
\draw[->] (K1) -- node[left] {${\phi_{12}}$} (PK1);
\draw[->] (K2) -- node[right] {${\phi_{22}}$} (PK2);
\draw[->] (PK1) -- node[left] {${\phi_{13}}$} (P1);
\draw[->] (PK2) -- node[right] {${\phi_{23}}$} (P2);
\end{tikzpicture}
\end{center}
\caption{$\theta' \leq \hat{\theta}' \leq \tilde{\theta} \leq \theta$}\label{fig:sub-decom}
\end{figure}
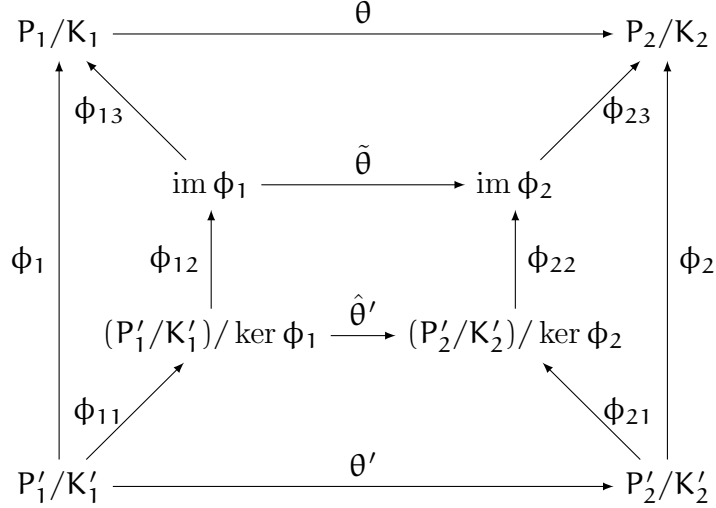
 By Proposition
  \ref{secdecom}, $\im \phi_i = P'_i K_i/K_i$ and
  $(P_i'/K_i')/\ker \phi_i \cong P'_i/(P'_i \cap K_i)$.
\end{proof}

\begin{cor} \label{cor:decompose-subgroups}
  The partial order $\leq$ on
  $\Sub{G_1 \times G_2}$ is a product of three relations:
  \begin{align*}
    {\leq} = {\leq}_{K} \circ {\leq}_{P/K} \circ {\leq}_{P}\text.
  \end{align*}
  Moreover, if $A(R)$ denotes the incidence matrix of the relation $R$, the
  stronger property
\begin{align*}
  A(\leq) = A(\leq_{K}) \cdot A(\leq_{P/K}) \cdot A(\leq_{P})
\end{align*}
also holds.
\end{cor}

\begin{proof}
  Like Corollary~\ref{ince}, this follows from the uniqueness of the
  intermediate subgroups in Theorem~\ref{thm:intermediate-subgroups}.
\end{proof}


\begin{lem} \label{la:corr}
Let $(\theta_i: P_i / K_i \rightarrow U) \in \Mor{G_i}(U)$, $i = 1, 2$ and $L = \Pi (\theta_1, \theta_2)$. Then
\begin{enumerate}
\item the set $\{L' \leq G_1 \times G_2 : L' \leq_K L \}$
  is in an order preserving bijective correspondence with the subgroups of $U$;
\item the set $\{L' \leq G_1 \times G_2 : L \leq_P L' \}$
  is in an order preserving bijective correspondence with the quotients of $U$;
\item the set $\{L' : L' \leq_{P/K} L \}$ is in an order
  preserving bijective correspondence with
  $\{ (P_1', K_1') : (P_1', K_1') \leq_{P/K} (P_1, K_1) \} \times \{ (P_2', K_2') : (P_2', K_2') \leq_{P/K}
  (P_2, K_2) \}$;
\item the set $\{L' : L \leq_{P/K} L' \}$
  is in an order preserving bijective correspondence with
 $\{ (P_1', K_1') : (P_1', K_1') \geq_{P/K} (P_1, K_1) \}
\times \{ (P_2', K_2') : (P_2', K_2') \geq_{P/K} (P_2, K_2) \}$.
\end{enumerate}
\end{lem}

\begin{proof}
  This follows from Lemma~\ref{la:mor-corr} on the correspondences induced by
  a $U$-morphism, together with Proposition~\ref{prop:comparing-morphisms}
and Theorem~\ref{thm:intermediate-subgroups}.
\end{proof}

\subsection{Classes of Subgroups}
\label{sec:classes-subgroups}

The conjugacy classes of $U$-subgroups of $G_1 \times G_2$ can be described
as $\Aut(U)$-orbits of pairs of classes of $U$-morphisms.

\begin{thm} \label{thm:cc-subgroups}
  Let $U \sqsubseteq G_i$, $i = 1,2$.
  \begin{enumerate}
  \item[(i)] $\Sub{G_1 \times G_2}(U)/(G_1 \times G_2)$
    is the disjoint union of sets
    \begin{align*}
      \Mor{G_1}^{P_1,K_1}(U)/G_1 \times_{\Aut(U)} (\Mor{G_2}^{P_2,K_2}(U)/G_2)^{\op},
    \end{align*}
    one for each pair of section classes $[P_i, K_i]_{G_i} \in \Sec{G_i}(U)/G_i$.
  \item[(ii)] Let $\theta_i \colon P_i/K_i \to U$ be $U$-morphisms of $G_i$, $i = 1,2$.  Then
    \begin{align*}
 \Mor{G_1}^{P_1,K_1}(U)/G_1 \times_{\Aut(U)} (\Mor{G_2}^{P_2,K_2}(U)/G_2)^{\op} = \{[\theta_1 d \theta_2^{-1}]_{G_1 \times G_2}: d \in D_{\theta_1,\theta_2}\}\text,
    \end{align*}
    where $D_{\theta_1,\theta_2}$ is a transversal of the
    $(A_{\theta_1},A_{\theta_2})$-double cosets in $\Aut(U)$.
  \end{enumerate}
\end{thm}

\begin{proof}
  (i) As
  $\Sub{G_1 \times G_2}(U) = \Mor{G_1}(U) \times_{\Aut(U)} \Mor{G_2}(U)$, the
  $(G_1 \times G_2)$-conjugacy classes of $U$-subgroups of $G_1 \times G_2$
  are $\Aut(U)$-orbits on the direct product
  $\Mor{G_1}(U)/G_1 \times \Mor{G_2}(U)/G_2$.  By
  Proposition~\ref{actonmorph}(i), this direct product is the disjoint union
  of $\Aut(U)$-invariant direct products
  $\Mor{G_1}^{P_1,K_1}(U)/G_1 \times \Mor{G_2}^{P_2,K_2}(U)/G_2$, one for
  each choice of $G_i$-classes of sections
  $[P_i, K_i]_{G_i} \in \Sec{G_i}(U)/G_i$, $i = 1,2$.

  (ii) Let $\alpha_i \in \Aut(U)$, $i = 1,2$.  Note first that the image of
  $[\theta_1 \alpha_1]_{G_1} \times [\theta_2 \alpha_2]_{G_2}$ under $\Pi$ is
  a $(G_1 \times G_2)$-conjugacy class of $U$-subgroups and that each
  $(G_1 \times G_2)$-class is of this form.  We show that the classes in
  $\Mor{G_1}^{P_1,K_1}(U)/G_1 \times \Mor{G_2}^{P_2,K_2}(U)/G_2$ correspond
  to the $(A_{\theta_1}, A_{\theta_2})$-double cosets in $\Aut(U)$.  For
  this, let $\alpha_i' \in \Aut(U)$, $i = 1,2$, and assume that
  $\Pi([\theta_1 \alpha_1]_{G_1}, [\theta_2 \alpha_2]_{G_2}) = \Pi([\theta_1
  \alpha_1']_{G_1}, [\theta_2 \alpha_2']_{G_2})$.
  By Proposition~\ref{actonmorph}(ii), this is the case if and only if
  $\theta_1 A_{\theta_1} \alpha_1 \alpha_2^{-1} A_{\theta_2} \theta_2^{-1} =
  \theta_1 A_{\theta_1} \alpha_1' (\alpha_2')^{-1} A_{\theta_2}
  \theta_2^{-1}$,
  i.e., if $\alpha_1 \alpha_2^{-1}$ and $\alpha_1' (\alpha_2')^{-1}$ lie in
  the same $(A_{\theta_1}, A_{\theta_2})$-double coset.
\end{proof}

\begin{ex}\label{ex:s3-ccs}
  Let $G = S_3$.  For each $U$-morphism $\theta \colon P/K \to U$, we have
  $O_{\theta} = \Out(U)$.  Therefore, by Theorem~\ref{thm:cc-subgroups},
  there exists exactly one conjugacy class of subgroups for each pair of
  classes of isomorphic sections $(P_1, K_1)$, $(P_2, K_2)$.  A transversal
  $\{L_1, \dots, L_{22}\}$ of the $22$ conjugacy classes of subgroups of
  $G \times G$ can be labelled by pairs of sections as follows.
\begin{align*}
  \begin{array}{r|cccc}
    & (1,1) & (2,2) & (3,3) & (G,G) \\ \hline
(1,1) & L_1&L_2&L_3&L_4 \\
(2,2) & L_5&L_6&L_7&L_8 \\
(3,3) & L_9&L_{10}&L_{11}&L_{12} \\
(G,G) & L_{13}&L_{14}&L_{15}&L_{16}
  \end{array}
&&
   \begin{array}{r|cc|c|c}
 & (2,1) & (G,3) & (3,1) & (G,1) \\ \hline
(2,1) & L_{17} & L_{18} &&\\
(G,3) & L_{19} & L_{20} &&\\ \hline
(3,1) &       &       & L_{21} & \\ \hline
(G,1) & &  &  &  L_{22}
   \end{array}
\end{align*}
Here, a subgroup $L_i$ in row  $(P_1,K_1)$ and column $(P_2,K_2)$
has a Goursat isomorphism of the form $P_1/K_1 \to P_2/K_2$.
\end{ex}

The normalizer of a subgroup $\theta = \Pi(\theta_1, \theta_2)$ of
$G_1 \times G_2$, described as a quotient of two $U$-morphisms $\theta_i$,
can be described as the quotient of the automizers of the two $U$-morphisms.

\begin{thm}
  Let $U \sqsubseteq G_i$ and let
   $\theta_i \in \Mor{G_i}(U)$, for $i = 1,2$. Then
  \begin{align*}
    N_{G_1 \times G_2} (\Pi(\theta_1, \theta_2))
    = A_{G_1}(\theta_1) \ast A_{G_2}(\theta_2)^{\op}
  \end{align*}
\end{thm}

\begin{proof}
  For $i = 1, 2$, suppose that $\theta_i \colon P_i/K_i \to U$ and let
  $(\tilde{P}_i, \tilde{K}_i) = \Atm{G_i}(P_i, K_i)$.  Then
  $\Atm{G_i}(\theta_i) \colon \tilde{P}_i/ \tilde{K}_i \to A_{\theta_i} \leq
  \Aut(U)$
  is the automizer of $\theta_i$.  Let
  $\theta = \Pi(\theta_1, \theta_2) = \theta_1 \theta_2^{-1}$.  Then, on the
  one hand,
  \begin{align*}
    N_{G_1 \times G_2} (\theta)
    = \{(a_1, a_2) \in G_1 \times G_2 :
    \gamma_{a_1}^{-1} \theta \gamma_{a_2} = \theta\}
  \end{align*}
  consists of those elements $(a_1, a_2) \in \tilde{P}_1 \times \tilde{P}_2$
  which induce automorphisms
  $\alpha_i = \theta_i^{-1} \gamma_{a_i} \theta_i
    = (a_i \tilde{K}_i)^{\Atm{G_i}(\theta_i)} \in \Aut(U)$
  such that
  $\theta_1 \alpha_1^{-1} \alpha_2 \theta_2^{-1} = \theta_1 \theta_2^{-1}$,
  i.e., $\alpha_1 = \alpha_2$.

  On the other hand, by the Lemma \ref{lm:star_butter},
  $\Atm{G_1}(\theta_1) \ast \Atm{G_2}(\theta_2)^{op} = \Pi(\tilde{\theta}_1'
  (\tilde{\theta}_2'))$ where, for $i = 1,2$,
\begin{align*}
  \tilde{\theta}_i' \colon \tilde{P}_i'/ \tilde{K}_i \to \tilde{U}
\end{align*}
is the restriction of the isomorphism $\Atm{G_i}(\theta_i)$
to the preimage $\tilde{P}_i'/ \tilde{K}_i$ of $\tilde{U} = A_{\theta_1} \cap A_{\theta_2}$ in $\tilde{P}_i/\tilde{K}_i$.  Hence, as a subgroup of $G_1 \times G_2$, the product
 $\Atm{G_1}(\theta_1) \ast \Atm{G_2}(\theta_2)^{op}$ consists of
those elements $(a_1, a_2) \in \tilde{P}_1' \times \tilde{P}_2'$
with
$(a_1 \tilde{K}_1)^{\Atm{G_1}(\theta_1)} = (a_2 \tilde{K}_2)^{\Atm{G_2}(\theta_2)}$.

It follows that $A_{G_1}(\theta_1) \ast A_{G_2}(\theta_2)^{\op} = N_{G_1 \times G_2}(\theta)$, as desired.
\end{proof}

As an immediate consequence, we can determine the normalizer index
of a subgroup of $G_1 \times G_2$ in terms of  $U$-morphisms.

\begin{cor}\label{cor:nor-idx}
  Let $L = \Pi(\theta_1, \theta_2) \leq G_1 \times G_2$, for $\theta_i \colon P_i/K_i \to U$, $i = 1,2$.  Then
  \begin{align*}
    |N_{G_1 \times G_2}(L) : L| =
      |C_{\overline{N}_1}(\overline{P}_1)|\,
      |C_{\overline{N}_2}(\overline{P}_2)|\,
      |O_{\theta_1} \cap O_{\theta_2}|\,
      |Z(U)|^{-1}\text,
  \end{align*}
  where $\overline{N}_i = N_{G_i}(K_i)/K_i$ and $\overline{P}_i = P_i/K_i$, $i = 1,2$.
\end{cor}

\begin{proof}
  By Lemma~\ref{la:size-subgroup}, $|L| = |K_1|\,|K_2|\,|U|$.  With the
  notation from the preceding proof,
  $|N_{G_1 \times G_2}(L)| = |\tilde{K}_1|\,|\tilde{K}_2|\,|\tilde{U}|$. Thus
\begin{align*}
  |N_{G_1 \times G_2}(L) : L| =
    \frac{|\tilde{K}_1|\,|\tilde{K}_2|\,|\tilde{U}|}{|K_1|\,|K_2|\,|U|}\text.
\end{align*}
But $|\tilde{K}_i : K_i| = |C_{\overline{N}_i}(\overline{P}_i)|$, $i=1,2$, by Definition~\ref{def:sec-autom}.
Moreover, $|\tilde{U}| = |A_{\theta_1} \cap A_{\theta_2}| = |\Inn(U)|\, |O_{\theta_1} \cap O_{\theta_2}|$ and
$|U| = |\Inn(U)|\,|Z(U)|$.
\end{proof}


\section{Table of Marks}
\label{sec:tableofmarks}

We are now in a position to assemble the table of marks of $G_1 \times G_2$
from a collection of smaller class incidence matrices.

\begin{thm} \label{tm:tom}
Let $G_1$ and $G_2$ be finite groups. Then the table of marks of $G_1 \times G_2$ is
\[M (G_1 \times G_2) = D \cdot \CIM(\leq_K) \cdot \CIM (\leq_{P/K}) \cdot \CIM (\leq_P),\]
where $D$ is the diagonal matrix with entries $|N_{G_1 \times G_2}(L) : L|$,
for $L$ running over a transversal of the conjugacy classes of subgroups of
$G_1 \times G_2$.
\end{thm}

\begin{proof}
The proof is similar to that of Theorem \ref{cimsecdecom} in combination with Corollary \ref{cor:L'<=L}.
\end{proof}

In the remainder of this section, we determine the block diagonal structure
of each of the matrices $\CIM(\leq_K)$, $\CIM(\leq_{P/K})$ and $\CIM(\leq_P)$
of $G_1 \times G_2$.

\subsection{}

The class incidence matrix of the $(G_1 \times G_2)$-poset $(\Sub{G_1 \times G_2}, \leq_K)$ is a block diagonal matrix, with one block for each pair
$([K_1], [K_2])$ of conjugacy classes $[K_i]$ of subgroups of $G_i$, $i=1,2$.

\begin{thm}\label{marksamek}
  For $K_i \leq G_i$, $i=1,2$, denote by $\CIM_{K_1, K_2}$ the class
  incidence matrix of $N_{G_1}(K_1) \times N_{G_2}(K_2)$ acting on the
  subposet of
  $(\Sub{G_1 \times G_2}, \leq)$
consisting of those subgroups $L$ with bottom groups $k_i(L) = K_i$, $i=1,2$.  Then
\[
\CIM(\leq_K) = \bigoplus_{\substack{[K_i] \in \Sub{G_i}/G_i\\i = 1,2}} \CIM_{K_1, K_2}\text.
\]
\end{thm}

\begin{proof}
Let $X = \Sub{G_1 \times G_2}$ and $Y = \Sub{G_1} \times \Sub{G_2}$. We identify $X$ with $Z \subseteq X \times Y$, where $Z:= \{ (x, y) : (k_1 (x), k_2 (x)) = y \}$. Then Lemma \ref{la:orbits-of-pairs} yields a partition of the conjugacy classes of subgroups of $G_1 \times G_2$ indexed by $[K_i] \in \Sub{G_i}/G_i$, $i =1, 2$.
The stabilizer of $y = (K_1, K_2) \in Y$ is $N_{G_1} (K_1) \times N_{G_2} (K_2)$ and $Zy = \{ x \in X : ( k_1 ( x), k_2 (x))  = y \}$.

Let $L \leq G_1 \times G_2$ be such that $k_i(L) = K_i$, $i = 1,2$.  In order to count the $(G_1 \times G_2)$-conjugates of a subgroup $L' \leq G_1 \times G_2$
with bottom groups $k_i(L') = K_i$, $i = 1,2$,
above $L$ in the $\leq_K$-order, it suffices to note that
$L \leq_K (L')^g$ for some $g \in G_1 \times G_2$ if and only if
$L \leq_K (L')^g$ for some $g \in N_{G_1}(K_1) \times N_{G_2}(K_2)$.

Finally by the definition of $\leq_K$ there are no incidences between subgroups with different $K_i$, giving the block diagonal structure.
\end{proof}

\begin{ex}\label{ex:cimk}
  Let $G_1 = G_2 = S_3$.  Then $\CIM(\leq_K)$ is the block sum of the
  matrices $\CIM_{K_1, K_2}$ in
  the table below, with rows and columns labelled by the conjugacy classes
  of subgroups $K$ of $S_3$.  Within $\CIM_{K_1, K_2}$,
  the row label of a subgroup of the form $P_1/K_1 \to P_2/K_2$ is just $P_1 \to P_2$,
  for brevity.  The column labels are identical and have been omitted.
\[
\begin{array}{c|cccc}
K&1&2&3&S_3\\
\hline

1&
\tiny{\begin{array}{c|cccc}
1 \to 1&1&   \cdot&   \cdot&   \cdot\\
2 \to 2& 9&   1&   \cdot& \cdot \\
3 \to 3& 2&  \cdot&   1&   \cdot\\
S_3 \to S_3& 6&  2&  3&   1\\
  \hline
  \end{array}}&
  \tiny{\begin{array}{c|c}
  1 \to 2&1\\
  \hline
  \end{array}}&
  \tiny{\begin{array}{c|cc}
  1 \to 3&1&\cdot\\
  2 \to S_3&3&1\\
  \hline
  \end{array}}&
  \tiny{
  \begin{array}{c|c}
  1 \to S_3&1\\
  \hline
  \end{array}}\\

  2&
  \tiny{
  \begin{array}{c|c}
  2 \to 1&1\\
  \hline
  \end{array}}&
  \tiny{
  \begin{array}{c|c}
  2 \to 2&1\\
  \hline
  \end{array}}&
  \tiny{
  \begin{array}{c|c}
  2 \to 3&1\\
  \hline
  \end{array}}&
  \tiny{
  \begin{array}{c|c}
  2 \to S_3&1\\
  \hline
  \end{array}}\\

  3&
  \tiny{
  \begin{array}{c|cc}
  3 \to 1& 1& \cdot\\
  S_3 \to 2&3&1\\
  \hline
  \end{array}}&
  \tiny{
  \begin{array}{c| c}
  3 \to 2&1\\
  \hline
  \end{array}}&
  \tiny{
  \begin{array}{c| c c}
  3 \to 3&1&\cdot\\
  S_3 \to S_3&1&1\\
  \hline
  \end{array}}&
  \tiny{ \begin{array}{c|c}
  3 \to S_3&1\\
  \hline
  \end{array}}\\

  S_3&
  \tiny{
  \begin{array}{c|c}
  S_3 \to 1&1\\
  \hline
  \end{array}}&
  \tiny{  \begin{array}{c| c}
  S_3 \to 2&1\\
  \hline
  \end{array}}&
  \tiny{
  \begin{array}{c|c}
  S_3 \to 3&1\\
  \hline
  \end{array}}&
  \tiny{
  \begin{array}{c|c}
  S_3 \to S_3&1\\
  \hline
  \end{array}}\\
  \end{array}
\]
\end{ex}

\subsection{}

The class incidence matrix of the $(G_1 \times G_2)$-poset $(\Sub{G_1 \times G_2}, \leq_{P/K})$ is a block diagonal matrix, with one block for each group
$U \subseteq G_i$, $i = 1,2$, up to isomorphism.

\begin{defn}
For a finite group $G$ and finite $G$-sets $X_1$ and $X_2$,
let $A_i$ be a square matrix with rows and columns labelled by
$X_i$, $i = 1,2$.  The action of $G$ on $X_1 \times X_2$ permutes
the rows and columns of the Kronecker product
$A_1 \otimes A_2$.  If the matrices
$A_1$ and $A_2$ are compatible with the $G$-action
then so is their Kronecker product, and we define
\[
A_1 \otimes_G A_2 := R(G) \cdot (A_1 \otimes A_2) \cdot C(G)\text,
\]
where the row summing and column picking matrices $R(G)$ and $C(G)$
have been constructed as in Lemma~\ref{classin}, with respect to the $G$-orbits
on $X_1 \times X_2$.
\end{defn}

For $U \sqsubseteq G_i$, consider the class incidence matrices
$A_i = \CIM_U^{G_i}(\leq)$ of the
$G_i$-posets $\Mor{G_i}(U)$, $i = 1,2$, from
Section~\ref{sec:partial-order-morphism-classes}.
By Proposition~\ref{pro:cim-mor}, these matrices, and hence
$A_1 \otimes A_2$, are
compatible with the action of $\Out(U)$ on their rows and
columns.

\begin{thm} \label{sameu}
We have
\[
\CIM (\leq_{P/K}) = \bigoplus_{U \sqsubseteq G_1, G_2} \CIM^{G_1}_{U}(\leq) \otimes_{\Out(U)} \CIM^{G_2}_{U}(\leq)\text,
\]
where, for $U \sqsubseteq G_i$, $\CIM^{G_i}_{U} (\leq)$ is  the class incidence matrix of the $G_i$-poset $\Mor{G_i}(U)$, $i = 1,2$.
\end{thm}

\begin{proof}
  Let $L' = (\theta': P_1'/K_1' \to P_2'/K_2')$ be a subgroup of
  $G_1 \times G_2$ with Goursat type $U$ and select $U$-morphisms $\theta_1'$
  and $\theta_2'$ such that $\Pi(\theta_1', \theta_2') = \theta'$. By Lemma
  \ref{la:corr} (iv) the subgroups $L$ of $G_1 \times G_2$ with
  $L \geq_{P/K} L'$ correspond to pairs of sections
  $(P_i, K_i) \geq_{P/K} (P_i', K_i')$, $i = 1,2$.  For each such section
  $(P_i, K_i)$, set $\theta_i = \phi_i^{-1} \theta_i'$, where
  $\phi_i \colon P_i'/K_i' \to P_i/K_i$ is the canonical isomorphism.  By
  Proposition~\ref{prop:comparing-morphisms}, $\theta_i: P_i/K_i \to U$ is the unique
  $U$-morphism in $\Mor{G_i}^{P_i,K_i}(U)$ with $\theta_i' \leq \theta_i$.
  The number of conjugates $L^x$ of a subgroup $L$ of $G_1 \times G_2$ with
  $L^x \geq_{P/K} L'$ is thus equal to the number of pairs
  $(\theta_1, \theta_2) \in \Mor{G_1}(U) \times \Mor{G_2}(U)$ with
  $\theta_i' \leq \theta_i$ such that $\Pi(\theta_1, \theta_2)$ is a
  conjugate of $L$ in $G_1 \times G_2$.

  If $L = \Pi(\theta_1, \theta_2)$ for $\theta_i \in \Mor{G_i}^{P_i,K_i}(U)$
  then, by Theorem \ref{thm:cc-subgroups}, the set of all pairs of
  $U$-morphisms mapping to a conjugate of $L$ under $\Pi$ is the
  $\Out(U)$-orbit of $[\theta_1]_{G_1} \times [\theta_2]_{G_2}$ in
  $\Mor{G_1}^{P_1,K_1}(U)/G_1 \times \Mor{G_2}^{P_2,K_2}(U)/G_2$.  The number
  of $G_i$-conjugates of $\theta_i$ above $\theta_i'$ is given by the entry
  $\mathbf{a}(\theta_i, \theta_i')$ of the class incidence matrix
  $\CIM^{G_i}_U$.  By Proposition~\ref{actonmorph}, the $\Aut(U)$-set
  $\Mor{G_i}^{P_i,K_i}(U)/G_i$ is isomorphic to $\Aut(U)/A_{\theta_i}$. Hence
  \begin{align*}
    \# \{L^x \geq_{P/K} L' : x \in G_1 \times G_2\} =
    \sum_{\alpha \in T_{\theta_1,\theta_2}} \mathbf{a}(\theta_1 \alpha, \theta_1') \, \mathbf{a}(\theta_2 \alpha, \theta_2')\text,
  \end{align*}
  where $T_{\theta_1,\theta_2}$  is a transversal  of the right
  cosets of $A_{\theta_1} \cap A_{\theta_2}$ in $\Aut(U)$.  As
  $T_{\theta_1,\theta_2}$ can also be used to represent the right
  cosets of $O_{\theta_1} \cap O_{\theta_2}$ in $\Out(U)$, the same number
  appears as the $L, L'$-entry of the matrix
  $\CIM^{G_1}_{U}(\leq) \otimes_{\Out(U)} \CIM^{G_2}_{U}(\leq)$.
\end{proof}

\begin{ex}\label{ex:cimu}
  Let $G_1 = G_2 = S_3$.  Then $\CIM(\leq_{P/K})$ is the block sum of the
  following matrices
  $\CIM^{G_1}_{U}(\leq) \otimes_{\Out(U)} \CIM^{G_2}_{U}(\leq)$.  As
  $\Out(U)$ here acts trivially, the matrices are simply the Kronecker
  squares of the matrices $\CIM^{G}_{U}(\leq)$ in
  Example~\ref{ex:s3-PKsec}.  The column labels are identical to the row
  labels and have been omitted.
\[
\begin{array}{c|c}
U & \CIM^{G_1}_{U}(\leq) \otimes_{\Out(U)} \CIM^{G_2}_{U}(\leq) \\
\hline
1&
\tiny{\begin{array}{c|cccccccccccccccc}
 1/1 \to 1/1&1&  \cdot&  \cdot&  \cdot&  \cdot&  \cdot&  \cdot&  \cdot&  \cdot&\cdot&  \cdot&  \cdot&  \cdot&  \cdot&  \cdot&  \cdot \\
1/1 \to 2/2& 3&  1&  \cdot&  \cdot&  \cdot&  \cdot&  \cdot&  \cdot&  \cdot&  \cdot&  \cdot&  \cdot&  \cdot&  \cdot&  \cdot&  \cdot\\
 1/1 \to 3/3&   1&  \cdot&  1&  \cdot&  \cdot&  \cdot&  \cdot&  \cdot&  \cdot&  \cdot&  \cdot&  \cdot&  \cdot&  \cdot&  \cdot&  \cdot\\
 1/1 \to S_3/S_3& 1&  1&  1&  1&  \cdot&  \cdot&  \cdot&  \cdot&  \cdot&  \cdot&  \cdot&  \cdot&  \cdot&  \cdot&  \cdot&  \cdot\\
  2/2 \to 1/1& 3&  \cdot&  \cdot&  \cdot&  1&  \cdot&  \cdot&  \cdot&  \cdot&  \cdot&  \cdot&  \cdot&  \cdot&  \cdot&  \cdot&  \cdot\\
   2/2 \to 2/2& 9&  3&  \cdot&  \cdot&  3&  1&  \cdot&  \cdot&  \cdot&  \cdot&  \cdot&  \cdot&  \cdot&  \cdot&  \cdot&  \cdot\\
   2/2 \to 3/3& 3&  \cdot&  3&  \cdot&  1&  \cdot&  1&  \cdot&  \cdot&  \cdot&  \cdot&  \cdot&  \cdot&  \cdot&  \cdot&  \cdot\\
   2/2 \to S_3/S_3& 3&  3&  3&  3&  1&  1&  1&  1&  \cdot&  \cdot&  \cdot&  \cdot&  \cdot&  \cdot&  \cdot&  \cdot\\
   3/3 \to 1/1& 1&  \cdot&  \cdot&  \cdot&  \cdot&  \cdot&  \cdot&  \cdot&  1&  \cdot&  \cdot&  \cdot&  \cdot&  \cdot&  \cdot&  \cdot\\
   3/3 \to 2/2& 3&  1&  \cdot&  \cdot&  \cdot&  \cdot&  \cdot&  \cdot&  3&  1&  \cdot&  \cdot&  \cdot&  \cdot&  \cdot&  \cdot\\
   3/3 \to 3/3& 1&  \cdot&  1&  \cdot&  \cdot&  \cdot&  \cdot&  \cdot&  1&  \cdot&  1&  \cdot&  \cdot&  \cdot&  \cdot&  \cdot\\
   3/3 \to S_3/S_3& 1&  1&  1&  1&  \cdot&  \cdot&  \cdot&  \cdot&  1&  1&  1&  1&  \cdot&  \cdot&  \cdot&  \cdot\\
   S_3/S_3 \to 1/1& 1&  \cdot&  \cdot&  \cdot&  1&  \cdot&  \cdot&  \cdot&  1&  \cdot&  \cdot&  \cdot&  1&  \cdot&  \cdot&  \cdot\\
   S_3/S_3 \to 2/2& 3&  1&  \cdot&  \cdot&  3&  1&  \cdot&  \cdot&  3&  1&  \cdot&  \cdot&  3&  1&  \cdot&  \cdot\\
   S_3/S_3 \to 3/3& 1&  \cdot&  1&  \cdot&  1&  \cdot&  1&  \cdot&  1&  \cdot&  1&  \cdot&  1&  \cdot&  1&  \cdot\\
   S_3/S_3 \to S_3/S_3& 1&  1&  1&  1&  1&  1&  1&  1&  1&  1&  1&  1&  1&  1&  1&  1\\
   \hline
   \end{array}}\\
2&\tiny{
 \begin{array}{c|cccc}
  2/1 \to 2/1&1&  \cdot&  \cdot&  \cdot \\
  2/1 \to S_3/3&1&  1&  \cdot&  \cdot \\
  S_3/3 \to 2/1&1&  \cdot& 1&  \cdot\\
  S_3/3 \to S_3/3&1&  1&  1&  1 \\
  \hline
  \end{array}} \\
3&\tiny{
  \begin{array}{c|c}
  3/1 \to 3/1&1\\
  \hline
  \end{array}}\\
S_3&\tiny{
  \begin{array}{c|c}
  S_3/1 \to S_3/1&1\\
  \hline
  \end{array}}\\
 \end{array}\]
\end{ex}

\begin{ex}\label{ex:a5-2}
Continuing Example~\ref{ex:a5-1} for $G = A_5$ and $U = 3$, we have
\begin{align*}
  \CIM^G_U(\leq) &= \tiny\left(\arraycolsep3pt
  \begin{array}{ccc}
    1&&\\
    1&1&\\
    1&\cdot&1\\
  \end{array}
\right)\text,&
\CIM^G_U(\leq) \otimes_{\Out(U)} \CIM^G_U(\leq) &= \tiny\left(\arraycolsep3pt
  \begin{array}{ccccc}
           1&&&&\\
            2&1&&&\\
            2&\cdot&1&&\\
            2&1&1&1&\\
            2&1&1&\cdot&1
  \end{array}
\right)\text,
\end{align*}
illustrating the effect of a non-trivial $\Out(U)$-action.
\end{ex}

\subsection{}

The class incidence matrix of the $(G_1 \times G_2)$-poset $(\Sub{G_1 \times G_2}, \leq_P)$ is a block diagonal matrix, with one block for each pair
$([P_1], [P_2])$ of conjugacy classes $[P_i]$ of subgroups of $G_i$, $i=1,2$.

\begin{thm} \label{samep}
For $P_i \leq G_i$, $i = 1,2$, denote by $\CIM_{P_1, P_2}$ the class incidence matrix of $N_{G_1}(P_1) \times N_{G_2}(P_2)$ acting on the sub poset
of $(\Sub{G_1 \times G_2}, \leq)$ consisting of those subgroups $L$ with
$p_i(L) = P_i$, $i = 1,2$.  Then
\[
\CIM(\leq_P) = \bigoplus_{\substack{P_i \in \Sub{G_i} /G_i,\\i =1,2}} \CIM_{P_1, P_2}\text,
\]
\end{thm}

\begin{proof}
  Similar to the proof of Theorem~\ref{marksamek}, with
  $X = \Sub{G_1 \times G_2}$, $Y = \Sub{G_1} \times \Sub{G_2}$
  $Z = \{(x, y) : (p_1 (x), p_2 (x)) = y\} \subseteq X \times Y$.
\end{proof}

\begin{ex}\label{ex:cimp}
  Again we let $G_1 = G_2 = S_3$.  Then $\CIM(\leq_P)$ is the block sum of the matrices $\CIM_{P_1,P_2}$ in
  the table below, with rows and columns labelled by the conjugacy classes
  of subgroups $P$ of $S_3$. Similar to Example~\ref{ex:cimk}, within $\CIM_{P_1,P_2}$,
  the row label of a subgroup of the form $P_1/K_1 \to P_2/K_2$ is just $K_1 \to K_2$,
  for brevity.  The column labels are identical and have been omitted.
\[
\begin{array}{c|cccc}
P&1&2&3&S_3\\
\hline
1&
\tiny{
  \begin{array}{c|c}
 1 \to 1&1\\
 \hline
 \end{array}}&
\tiny{\begin{array}{c|c}
 1 \to 2&1\\
 \hline
 \end{array}}&
 \tiny{\begin{array}{c|c}
 1 \to 3&1\\
 \hline
 \end{array}}&
 \tiny{\begin{array}{c|c}
 1 \to S_3&1\\
 \hline
 \end{array}}\\

 2&
 \tiny{\begin{array}{c|c}
 2 \to 1&1\\
 \hline
 \end{array}}&
 \tiny{\begin{array}{c|cc}
 1 \to 1&1&\cdot\\
 2 \to 2&1&1\\
 \hline
 \end{array}}&
 \tiny{\begin{array}{c|c}
 2 \to 3&1\\
 \hline
 \end{array}}&
 \tiny{\begin{array}{c|cc}
 1 \to 3&1&\cdot\\
 2 \to S_3&1&1\\
 \hline
 \end{array}}\\

 3&
  \tiny{\begin{array}{c|c}
 3 \to 1&1\\
 \hline
 \end{array}}&
 \tiny{\begin{array}{c|c}
 3 \to 2&1\\
 \hline
 \end{array}}&
 \tiny{\begin{array}{c|cc}
 1 \to 1&1&\cdot\\
 3 \to 3&1&1\\
 \hline
 \end{array}}&
 \tiny{\begin{array}{c|c}
 3 \to S_3&1\\
 \hline
 \end{array}}\\

S_3&
\tiny{\begin{array}{c|c}
S_3 \to 1&1\\
\hline
\end{array}}&
\tiny{\begin{array}{c|cc}
3 \to 1&1&\cdot\\
S_3 \to 2&1&1\\
\hline
\end{array}}&
\tiny{\begin{array}{c|c}
S_3 \to 3&1\\
\hline
\end{array}}&
\tiny{\begin{array}{c|ccc}
1 \to 1&1&\cdot&\cdot\\
3 \to 3&1&1&\cdot\\
S_3 \to S_3&1&1&1\\
\hline
\end{array}}\\
\end{array}
\]
\end{ex}

\begin{ex}
  Combining the matrices from Examples~\ref{ex:cimk}, \ref{ex:cimu}
  and~\ref{ex:cimp} according to Theorem~\ref{tm:tom}, yields the table of
  marks $M$ of $S_{3} \times S_{3}$ with rows and columns sorted by section
  size, as in Example~\ref{ex:cimu}:
\[
M = \footnotesize \left(\begin{array}{cccccccccccccccc|cccc|c|c}
   36&  \cdot&   \cdot&  \cdot&  \cdot&  \cdot&  \cdot&  \cdot&   \cdot&  \cdot&  \cdot&  \cdot&  \cdot&  \cdot&  \cdot&  \cdot& \cdot&  \cdot&  \cdot&  \cdot& \cdot&  \cdot\\
   18&  6&   \cdot&  \cdot&  \cdot&  \cdot&  \cdot&  \cdot&   \cdot&  \cdot&  \cdot&  \cdot&  \cdot&  \cdot&  \cdot&  \cdot& \cdot&  \cdot&  \cdot&  \cdot& \cdot& \cdot\\
   12&  \cdot& 12&  \cdot&  \cdot&  \cdot&  \cdot&  \cdot&   \cdot&  \cdot&  \cdot&  \cdot&  \cdot&  \cdot&  \cdot&  \cdot&  \cdot&  \cdot&  \cdot&  \cdot& \cdot&   \cdot\\
   6&  6&  6& 6&  \cdot&  \cdot&  \cdot&  \cdot&   \cdot&  \cdot&  \cdot&  \cdot&  \cdot&  \cdot&  \cdot&  \cdot& \cdot&  \cdot&  \cdot&  \cdot& \cdot& \cdot\\
   18&  \cdot&   \cdot&  \cdot& 6&  \cdot&  \cdot&  \cdot&   \cdot&  \cdot&  \cdot&  \cdot&  \cdot&  \cdot&  \cdot&  \cdot& \cdot&  \cdot&  \cdot&  \cdot&  \cdot&  \cdot\\
    9&  3&   \cdot&  \cdot&  3& 1&  \cdot&  \cdot&  \cdot&  \cdot&  \cdot&  \cdot&  \cdot&  \cdot&  \cdot&  \cdot& 1&  \cdot&  \cdot&  \cdot& \cdot&  \cdot\\
    6&  \cdot&  6&  \cdot& 2&  \cdot& 2&  \cdot&   \cdot&  \cdot&  \cdot&  \cdot&  \cdot&  \cdot&  \cdot&  \cdot& \cdot&  \cdot&  \cdot&  \cdot& \cdot&  \cdot\\
   3& 3&  3& 3& 1& 1& 1& 1&   \cdot&  \cdot&  \cdot&  \cdot&  \cdot&  \cdot&  \cdot&  \cdot& 1& 1&  \cdot&  \cdot& \cdot& \cdot\\
  12& \cdot&  \cdot&  \cdot&  \cdot&  \cdot&  \cdot&  \cdot&12&  \cdot&  \cdot&  \cdot&  \cdot& \cdot&  \cdot&  \cdot& \cdot&  \cdot&  \cdot&  \cdot& \cdot&  \cdot\\
    6& 2&  \cdot& \cdot&  \cdot&  \cdot&  \cdot&  \cdot&  6& 2&  \cdot&  \cdot&  \cdot&  \cdot&  \cdot&  \cdot& \cdot&  \cdot&  \cdot&  \cdot& \cdot& \cdot\\
    4&  \cdot&  4& \cdot& \cdot&  \cdot&  \cdot&  \cdot&  4 & \cdot& 4&  \cdot&  \cdot&  \cdot&  \cdot&  \cdot& \cdot&  \cdot&  \cdot&  \cdot&  4&  \cdot\\
    2& 2&  2& 2& \cdot& \cdot&  \cdot&  \cdot&  2& 2& 2& 2&  \cdot&  \cdot&  \cdot&  \cdot& \cdot& \cdot&  \cdot&  \cdot& 2& \cdot\\
    6& \cdot&   \cdot&  \cdot& 6& \cdot&  \cdot& \cdot&  6&  \cdot&  \cdot&  \cdot& 6&  \cdot&  \cdot&  \cdot& \cdot& \cdot&  \cdot&  \cdot& \cdot&  \cdot\\
    3& 1&  \cdot&  \cdot& 3& 1& \cdot& \cdot&  3& 1&  \cdot&  \cdot& 3& 1&  \cdot&  \cdot& 1& \cdot& 1&  \cdot& \cdot&  \cdot\\
    2& \cdot&  2&  \cdot& 2& \cdot& 2& \cdot&  2& \cdot&  2&  \cdot& 2& \cdot&  2&  \cdot& \cdot&  \cdot& \cdot&  \cdot&  2& \cdot\\
    1& 1& 1& 1&1& 1& 1& 1& 1& 1& 1& 1& 1& 1& 1& 1&1& 1& 1& 1&  1&  1\\
    \hline
  18&  \cdot&   \cdot&  \cdot&  \cdot&  \cdot&  \cdot&  \cdot&  \cdot& \cdot&  \cdot&  \cdot& \cdot&  \cdot&  \cdot&  \cdot& 2&  \cdot&  \cdot& \cdot& \cdot&  \cdot\\
    6&  \cdot&  6&  \cdot&  \cdot&  \cdot&  \cdot&  \cdot&  \cdot& \cdot&  \cdot&  \cdot& \cdot&  \cdot&  \cdot&  \cdot& 2& 2&  \cdot& \cdot& \cdot& \cdot\\
    6&  \cdot&  \cdot&   \cdot&  \cdot&  \cdot&  \cdot&  \cdot& 6& \cdot&  \cdot&  \cdot& \cdot&  \cdot&  \cdot&  \cdot&  2&  \cdot& 2& \cdot& \cdot& \cdot\\
    2&  \cdot& 2&   \cdot&  \cdot&  \cdot&  \cdot&  \cdot& 2& \cdot& 2&  \cdot& \cdot&  \cdot&  \cdot&  \cdot& 2& 2& 2 & 2&2& 2\\
    \hline
  12&  \cdot&  \cdot&   \cdot&  \cdot&  \cdot&  \cdot&  \cdot&  \cdot& \cdot&  \cdot&   \cdot& \cdot& \cdot&  \cdot&  \cdot&  \cdot&   \cdot&  \cdot&  \cdot&  6&   \cdot\\
   \hline
    6&  \cdot&  \cdot&   \cdot&  \cdot&  \cdot&  \cdot&  \cdot&  \cdot& \cdot&  \cdot&   \cdot& \cdot& \cdot&  \cdot&  \cdot& 2&  \cdot&  \cdot&  \cdot&  3& 1
 \end{array}\right)
\]
\end{ex}


\section{The Double Burnside Algebra of $S_3$}
\label{sec:burnside}

As an application of the ideas from previous sections we now
construct a mark homomorphism for the rational double Burnside algebra of $G = S_3$.

\subsection{The Double Burnside Ring.}
\label{sec:double-burnside}

Let $G$, $H$ and $K$ be finite groups.  The Grothendieck group of
the category of finite $(G, H)$-bisets is denoted by $B(G, H)$.
If $(G, H)$-bisets are identified with $G \times H$-sets,
then the abelian group $B(G, H)$ is identified with the Burnside group
$B(G \times H)$, and hence the transitive bisets $[G \times H/L]$,
where $L$ runs through a transversal of the conjugacy classes of
subgroups of $G \times H$, form a $\Z$-basis of $B(G, H)$.

There is a bi-additive map from   $B(G, H) \times B(H, K)$ to
$B(G, K)$ given by
\begin{align*}
 ([X], [Y]) &\mapsto [X] \cdot_{H} [Y] = [X \times_H Y]\text.
\end{align*}
Multiplication of transitive bisets is described by the following
Mackey-formula.
\begin{prop}[\protect{\cite[2.3.24]{Bouc2010}}]\label{prop:double-mackey}
  Let $L \leq G \times H$ and $M \leq H \times K$.  Let $X \subseteq H$ be a transversal
of the
$(p_2(L),p_1(M))$-double cosets in $H$.
Then
  \begin{align*}
    [(G \times H)/L] \cdot_H [(H \times K)/M] =
\sum_{x \in X} [(G \times K)/(L^{(1, x)} \ast M)]
  \end{align*}
\end{prop}
With this multiplication, in particular,  $B(G, G)$ is a ring, the
\emph{double Burnside ring} of~$G$.

The rational double Burnside algebra $\Q B(G, G) = \Q \otimes_{\Z} B(G, G)$
is known to be semisimple, if and only if $G$ is cyclic \cite[Proposition~6.1.7]{Bouc2010}.  Little more is
known about the structure of $\Q B(G, G)$ in general.

\subsection{A Mark Homomorphism for the Double Burnside Ring of $S_3$.}
\label{sec:ghost-s3}

For the ordinary Burnside ring $B(G)$, the table of marks of $G$
is the matrix of the mark isomorphism $\beta_G \colon \Q B(G) \to \Q^r$
between the rational Burnside algebra and its ghost algebra.
It is an open question, whether there exist equivalent constructions
of ghost algebras and mark homomorphims for the double Burnside ring.
Boltje and Danz~\cite{BoltjeDanz2013} have investigated the role
of the table of marks of the direct product $G \times G$ in this context.
Here, we use the decomposition of the table of marks of $G \times G$
from Theorem~\ref{tm:tom} and the idea of transposing the $\leq_P$ part
from Section~\ref{sec:sect-latt-revis} in order to build a satisfying
ghost algebra for the group $G = S_3$.

For this purpose, we first set up a labelling of the natural basis of
$\Q B(G, G)$ as follows.
Set $I = \{1, \dots, 22\}$.
Let $\{L_i : i \in I\}$ be the conjugacy class representatives from
Example~\ref{ex:s3-ccs}.
Then the rational Burnside algebra $\Q B(G, G)$ has a $\Q$-basis
consisting of elements
$b_i = [G\times G/L_i]$, $i \in I$, and multiplication
defined by~\ref{prop:double-mackey}.

By Theorem~\ref{tm:tom}, the table of marks $M$ of $G \times G$ is a matrix product
\begin{align*}
  M = D_0 \cdot  \CIM(\leq_K) \cdot \CIM(\leq_{P/K}) \cdot \CIM(\leq_P)\text,
\end{align*}
of a diagonal matrix $D_0$ with entries $\left|N_{G \times G}(L_i) : L_i\right|$,
$i \in I$, and three class incidence matrices.
For our purpose, we now modify this product and set
\begin{align*}
  M' = \tfrac16\, D_0 \cdot  \CIM(\leq_K) \cdot \CIM(\leq_{P/K}) \cdot D_1 \cdot \CIM(\geq_P) \cdot D_2\text,
\end{align*}
where
\begin{align*}
D_1 &= \diag(1, 1, 1, 1, 1, 1, 1, 1, 2, 2, 2, 2, 6, 6, 6, 6, 1, 1, 6, 6, 1, 1)\text,\\
D_2 &= \diag(1, 1, 1, 1, 3, 3, 3, 3, 1, 1, 1, 1, 1, 1, 1, 1, 3, 3, 1, 1, 2, 6)\text,
\end{align*}
are diagonal matrices.  The resulting matrix is
\begin{align*} M' &= \tiny
\renewcommand{\arraystretch}{1.2}\arraycolsep=4pt
\left(\begin{array}{rrrr|rrrr|rrrr|rrrr|rr|rr|r|r}%
6&.&.&.&.&.&.&.&.&.&.&.&.&.&.&.&.&.&.&.&.&.\\%
3&1&.&.&.&.&.&.&.&.&.&.&.&.&.&.&.&.&.&.&.&.\\%
2&.&2&.&.&.&.&.&.&.&.&.&.&.&.&.&.&.&.&.&.&.\\%
1&1&1&1&.&.&.&.&.&.&.&.&.&.&.&.&.&.&.&.&.&.\\\hline%
3&.&.&.&3&.&.&.&.&.&.&.&.&.&.&.&.&.&.&.&.&.\\%
\nicefrac{3}{2}&\nicefrac{1}{2}&.&.&\nicefrac{3}{2}&\nicefrac{1}{2}&.&.&.&.&.&.&.&.&.&.&.&.&.&.&.&.\\%
1&.&1&.&1&.&1&.&.&.&.&.&.&.&.&.&.&.&.&.&.&.\\%
\nicefrac{1}{2}&\nicefrac{1}{2}&\nicefrac{1}{2}&\nicefrac{1}{2}&\nicefrac{1}{2}&\nicefrac{1}{2}&\nicefrac{1}{2}&\nicefrac{1}{2}&.&.&.&.&.&.&.&.&.&.&.&.&.&.\\\hline%
2&.&.&.&.&.&.&.&4&.&.&.&.&.&.&.&.&.&.&.&.&.\\%
1&\nicefrac{1}{3}&.&.&.&.&.&.&2&\nicefrac{2}{3}&.&.&.&.&.&.&.&.&.&.&.&.\\%
\nicefrac{2}{3}&.&\nicefrac{2}{3}&.&.&.&.&.&\nicefrac{4}{3}&.&\nicefrac{4}{3}&.&.&.&.&.&.&.&.&.&.&.\\%
\nicefrac{1}{3}&\nicefrac{1}{3}&\nicefrac{1}{3}&\nicefrac{1}{3}&.&.&.&.&\nicefrac{2}{3}&\nicefrac{2}{3}&\nicefrac{2}{3}&\nicefrac{2}{3}&.&.&.&.&.&.&.&.&.&.\\\hline%
1&.&.&.&3&.&.&.&2&.&.&.&6&.&.&.&.&.&.&.&.&.\\%
\nicefrac{1}{2}&\nicefrac{1}{6}&.&.&\nicefrac{3}{2}&\nicefrac{1}{2}&.&.&1&\nicefrac{1}{3}&.&.&3&1&.&.&.&.&.&.&.&.\\%
\nicefrac{1}{3}&.&\nicefrac{1}{3}&.&1&.&1&.&\nicefrac{2}{3}&.&\nicefrac{2}{3}&.&2&.&2&.&.&.&.&.&.&.\\%
\nicefrac{1}{6}&\nicefrac{1}{6}&\nicefrac{1}{6}&\nicefrac{1}{6}&\nicefrac{1}{2}&\nicefrac{1}{2}&\nicefrac{1}{2}&\nicefrac{1}{2}&\nicefrac{1}{3}&\nicefrac{1}{3}&\nicefrac{1}{3}&\nicefrac{1}{3}&1&1&1&1&.&.&.&.&.&.\\\hline%
3&.&.&.&.&1&.&.&.&.&.&.&.&.&.&.&1&.&.&.&.&.\\%
1&.&1&.&.&1&.&1&.&.&.&.&.&.&.&.&1&1&.&.&.&.\\\hline%
1&.&.&.&.&1&.&.&2&.&.&.&.&2&.&.&1&.&2&.&.&.\\%
\nicefrac{1}{3}&.&\nicefrac{1}{3}&.&.&1&.&1&\nicefrac{2}{3}&.&\nicefrac{2}{3}&.&.&2&.&2&1&1&2&2&.&.\\\hline%
2&.&.&.&.&.&.&.&.&.&2&.&.&.&.&.&.&.&.&.&2&.\\\hline%
1&.&.&.&.&1&.&.&.&.&1&.&.&.&.&.&1&.&.&.&1&1\\%
\end{array}\right)
\end{align*}
The matrix $M' = (m'_{ij})$ is obviously invertible,
hence there are unique elements $c_j \in \Q B(G, G)$, $j \in I$, such that
\begin{align*}
  b_i = \sum_{j \in I} m'_{ij} c_j\text,
\end{align*}
forming a new $\Q$-basis of $\Q B(G, G)$.

\begin{thm}
  Let $G = S_3$.
 Then the linear map $\beta'_{G \times G} \colon \Q B(G, G) \to \Q^{8 \times 8}$
 defined by
   \begin{align*}
 \beta'_{G \times G}\Bigl(\sum_{i \in I} x_i c_i\Bigr) = \left(
     \begin{array}{cccc|cc|c|c}
       x_1&x_2&x_3&x_4&.&.&.&.\\
       x_5&x_6&x_7&x_8&.&.&.&.\\
       x_9&x_{10}&x_{11}&x_{12}&.&.&.&.\\
       .&.&.&x_{22}&.&.&.&.\\ \hline
       .&.&.&.&x_{17}&x_{18}&.&.\\
       .&.&.&.&.&x_{22}&.&.\\ \hline
       .&.&.&.&.&.&x_{21}&.\\ \hline
       x_{13}&x_{14}&x_{15}&x_{16}&x_{19}&x_{20}&.&x_{22}
     \end{array}
 \right)\text,
   \end{align*}
 where $c_i \in \Q B(G, G)$ are defined as above, and $x_i \in \Q$, $i \in I$,
 is an injective homomorphism of algebras.
\end{thm}

\begin{proof}
This claim is validated by an explicit calculation, whose details we omit.
The general strategy is as follows.
For $i \in I$, let $C_i$ be the matrix of $c_i$ in the right regular representation of $\Q B(G, G)$ (computed with the help of the Mackey formula in Proposition~\ref{prop:double-mackey}).  Let $\equiv$ be the equivalence relation on
$I$ corresponding to the kernel of the map that sends the
conjugacy class of a subgroup $L = (P/K \to P'/K')$ to the conjugacy class
of the section $P'/K'$. Then $\equiv$ partitions $I$ as
\[
\{
\{1,5,9,13\},
\{2,6,10,14\},
\{3,7,11,15\},
\{4,8,12,16\},
\{17,19\},
\{18,20\},
\{21\},
\{22\}
\}\text.
\]
It turns out that all transposed matrices $C_i^T$ are compatible with the
equivalence $\equiv$ in the sense of
Section~\ref{sec:class-incidence-matrix}.  Hence, after choosing a
transversal of $\equiv$, and using the corresponding row summing and column
picking matrices $R(\equiv)$ and $C(\equiv)$, the map $\beta'$ defined by
\begin{align*}
  \beta'(c_i) = C(\equiv)^T \cdot C_i \cdot R(\equiv)^T\text,\quad i \in I\text,
\end{align*}
is independent of the choice of transversal.  In fact,
$\beta' = \beta'_{G\times G}$.  By Lemma~\ref{classin},
\begin{align*}
  \beta'(c_i c_k)
= C(\equiv)^T \cdot C_i \cdot C_k \cdot R(\equiv)^T
= \beta'(c_i) \cdot C(\equiv)^T \cdot C_k \cdot R(\equiv)^T
= \beta'(c_i) \cdot \beta'(c_k)\text,
\end{align*}
for $i, k \in I$,
showing that $\beta'_{G\times G} = \beta'$ is a homomorphism.
 Injectivity follows from a dimension count.
\end{proof}

It might be worth pointing out that the equivalence $\equiv$, and hence the
notion of compatibility and the map $\beta'$ depend on the basis used for the
matrices of the right regular representation.  In the case $G = S_3$, the
natural basis $\{b_i\}$ of $\Q B(G,G)$ also yields compatible matrices, but
the corresponding map $\beta'$ is not injective.  A base change under the
table of marks of $G \times G$ gives matrices which are not compatible.
Changing basis under the matrix product
$D_0 \cdot \CIM(\leq_K) \cdot \CIM(\leq_{P/K}) \cdot \CIM(\geq_P)$ yields
compatible matrices and an injective homomorphism like $M'$ does.  Our
matrices $\beta'_{G\times G}(c_i)$ have the added benefit of being normalized
and extremely sparse, exposing other representation theoretic properties of
the algebra $\Q B(G, G)$, such as the following.

\begin{cor}
Let $G = S_3$ and denote by $J$ the Jacobsen radical of the rational Burnside algebra $\Q B(G, G)$.
\begin{enumerate}
\item With $c_i$ as above, $\{c_i: i = 4,8,12,13,14,15,16,18,19,20\}$ is a
  basis of $J$.
\item
  $\Q B(G, G)/J \cong  \Q^{3 \times 3} \oplus \Q \oplus \Q \oplus \Q$.
\end{enumerate}
\end{cor}
The map $\beta'_{G \times G} \colon \Q B(G, G) \to \Q^{8\times 8}$ can be regarded as
a \emph{mark homomorphism}
for the double Burnside ring of $G = S_3$.  It assigns
to each $(G,G)$-biset a square matrix of rational marks.
For example, for
\begin{align*}
b_{20} &= [(G \times G)/L_{20}] \\&=
\tfrac{1}{3} c_{1} +\tfrac{1}{3} c_{3} + c_{6} + c_{8} +\tfrac{2}{3} c_{9} +\tfrac{2}{3} c_{11} +2 c_{14} +2 c_{16} + c_{17} + c_{18} +2 c_{19} +2 c_{20}
\end{align*}
we have
\begin{align*}
  \beta'_{G \times G}(b_{20}) &=
\small\left(
  \begin{array}{cccc|cc|c|c}
\nicefrac{1}{3}&.&\nicefrac{1}{3}&.&.&.&.&.\\
.&1&.&1&.&.&.&.\\
\nicefrac{2}{3}&.&\nicefrac{2}{3}&.&.&.&.&.\\
.&.&.&.&.&.&.&.\\\hline
.&.&.&.&1&1&.&.\\
.&.&.&.&.&.&.&.\\\hline
.&.&.&.&.&.&.&.\\\hline
.&2&.&2&2&2&.&.
  \end{array}
\right)\text,&
\end{align*}
and the image of
\begin{align*}
  b_{22} = [G \times G/G] = c_1 + c_6 + c_{11} + c_{17} + c_{21} + c_{22}
\end{align*}
is the identity matrix.

While the case $G = S_3$ provides only a small example, and the above
construction involves some ad hoc measures, we expect that for many if not
all finite groups $G$ a mark homomorphism for the rational double Burnside
algebra $\Q B(G, G)$ can be constructed in a similar way.  This will be the
subject of future research.


\bibliography{paper}
\bibliographystyle{amsplain}

\end{document}